\documentclass{article}
\usepackage{multicol,graphicx,color}
\usepackage{pslatex}
\usepackage{authblk}
\usepackage{amsthm}
\usepackage{amsmath}
\usepackage{amssymb}
\usepackage{latexsym}
\usepackage{lscape}
\usepackage{epsfig}
\usepackage{pstricks}
\usepackage{amsfonts}
\usepackage{mathrsfs}
\usepackage{mathrsfs}
\usepackage[
  hmarginratio={1:1},     
  vmarginratio={1:1},     
  textwidth=15cm,        
  textheight=21cm,
  heightrounded,          
]{geometry}

\usepackage{graphicx,color}
\usepackage[colorlinks]{hyperref}
\hypersetup{linkcolor=blue,citecolor=blue,filecolor=black,urlcolor=blue}

\theoremstyle{plain}
\newtheorem{theorem}{Theorem}

\newtheorem{lemma}[theorem]{Lemma}
\newtheorem{proposition}[theorem]{Proposition}
\theoremstyle{definition}
\newtheorem{definition}[theorem]{Definition}
\newtheorem{remark}[theorem]{Remark}

\numberwithin{equation}{section}
\numberwithin{theorem}{section}

\author[1]{Xiaojun Chang \footnote{changxj100@nenu.edu.cn}}
\author[2,3,4,5,6]{Vicen$\c{t}$iu D. R$\check{a}$dulescu\footnote{radulescu@inf.ucv.ro}} 

\author[1]{Yuxuan Zhang\footnote{zhangyx595@nenu.edu.cn}}  
\affil[1]{School of Mathematics and Statistics \& Center for Mathematics and Interdisciplinary Sciences,
 Northeast Normal University, Changchun 130024, Jilin,
PR China}
\affil[2]{Faculty of Applied Mathematics, AGH University of Krak\'ow, al. Mickiewicza 30, 30-059 Krak\'ow, Poland}
\affil[3]{Simion Stoilow Institute of Mathematics of the Romanian Academy, 21 Calea Grivi\c{t}ei, 010702 Bucharest, Romania}
\affil[4]{Department of Mathematics, University of Craiova, Street A.I. Cuza 13, 200585 Craiova, Romania}
\affil[5]{Brno University of Technology, Faculty of Electrical Engineering and Communication, Technick\'a 3058/10, Brno 61600, Czech Republic}
\affil[6]{School of Mathematics, Zhejiang Normal University, Jinhua, Zhejiang 321004, PR China}

\title{Solutions with prescribed mass for $L^2$-supercritical NLS equations under Neumann boundary conditions}
\date{}

\begin{document}

\maketitle

\begin{abstract}
\noindent In this paper, we investigate the following nonlinear Schr\"odinger equation with Neumann boundary conditions:
\begin{equation*}
\begin{cases}
	-\Delta u+ \lambda u= f(u)  & {\rm in} \,~ \Omega,\\
\displaystyle\frac{\partial u}{\partial \nu}=0  \, &{\rm on}\,~\partial \Omega
\end{cases}
\end{equation*}
coupled with a constraint condition:
    \begin{equation*}
         \int_{\Omega}|u|^2 dx=c,
    \end{equation*}
where $\Omega\subset \mathbb{R}^N(N\ge3)$ denotes a smooth bounded domain, $\nu$ represents the unit outer normal vector to $\partial \Omega$, $c$ is a positive constant, and $\lambda$ acts as a Lagrange multiplier. When the nonlinearity $f$ exhibits a general mass supercritical growth at infinity, we establish the existence of normalized solutions, which are not necessarily positive solutions and can be characterized as mountain pass type critical points of the associated constraint functional.
Our approach provides a uniform treatment of various nonlinearities, including cases such as $f(u)=|u|^{p-2}u$, $|u|^{q-2}u+ |u|^{p-2}u$, and $-|u|^{q-2}u+|u|^{p-2}u$, where $2<q<2+\frac{4}{N}<p< 2^*$. The result is obtained through a combination of a minimax principle with Morse index information for constrained functionals and a novel blow-up analysis for the NLS equation under Neumann boundary conditions.
\end{abstract}

\smallskip
{\small \noindent \text{\bf Key Words:} Normalized solutions; Nonlinear Schr\"odinger equation; 
	Bounded domains; Neumann boundary conditions; Variational methods.\\
	\text{\bf 2020 Mathematics Subject Classification:} 35A15, 35J20, 35J60, 35Q55}

\section{Introduction and main results}\label{intro}

This paper is concerned with the following Neumann problem for nonlinear Schr\"odinger equation(NLS):
\begin{equation}\label{schrodinger}
\begin{cases}
-\Delta u+ \lambda u= f(u)  & {\rm in} \,~ \Omega,\\
\displaystyle\frac{\partial u}{\partial \nu}=0  \, &{\rm on}\,~\partial \Omega
\end{cases}
\end{equation}
 with a mass constraint
 \begin{equation}\label{1.2}
 \int_\Omega |u|^2 dx=c,
 \end{equation}
where $\Omega\subset \mathbb{R}^N(N\ge3)$ is a smooth bounded domain, $c>0$ is a given constant, $\nu$ stands for the unit outer normal to $\partial \Omega$, and $\lambda$ serves as a Lagrange multiplier.

The features of problem \eqref{schrodinger}--\eqref{1.2} are the following:
\begin{itemize}
	\item [$(\text{i})$] The presence of the Neumann boundary condition is quite rare in the analysis of solutions with prescribed mass.
	
	\item [$(\text{ii})$] The reaction exhibits a general mass supercritical growth at infinity.
	
	\item [$(\text{iii})$] The proof relies on powerful techniques, such as Morse theory and a new blow-up analysis for the NLS equation.
	
	\item [$(\text{iv})$] The analysis presented in this paper can be extended to other classes of stationary problems, including  biharmonic elliptic equations and Schr\"odinger-Poisson equations. 
\end{itemize}

	The analysis of solutions with prescribed mass is particularly significant from a
physical point of view, in relationship with phenomena arising in nonlinear optics, the theory of water waves, etc. Indeed, solutions with prescribed $L^2$-norm are especially relevant since this quantity is
preserved along the time evolution. Moreover, the variational characterization of such solutions is often a strong help to analyze
their orbital stability and instability properties, see \cite{CL1,JJLV2022,JL2022,Soave-2020JDE,Soave-2020JFA}. 

The investigation of nonlinear Neumann problem \eqref{schrodinger} finds applications in various fields. One of the main motivations stems from the analysis of standing waves in the form $\psi(t,x) = e^{-i\lambda t}u(x)$, where $\lambda \in \mathbb{R}$ and $u:\mathbb{R}^N\to \mathbb{R}$, for the time-dependent nonlinear Schr\"odinger equation:
\begin{equation}\label{tNLS}
	i\partial_t \psi +\Delta \psi + g(|\psi|)\psi =0,~~(t, x)\in \mathbb{R}\times \Omega
\end{equation}
subject to Neumann boundary conditions. The model (\ref{tNLS}) plays an important role in nonlinear optics and Bose-Einstein condensates (see \cite{Agr2013, BC2013, FM2001,F2010,Malomed2008}).
 It is evident that  $\psi(t,x)$ is a solution to \eqref{tNLS} if and only $(u, \lambda)$ constitutes a coupled solution to \eqref{schrodinger} with $f(u)=g(|u|)u$.
We study equation \eqref{schrodinger} by searching for solutions satisfying the mass constraint $\int_{\Omega}|u|^2 dx=c$.
 In this context, $\lambda$ remains unknown and emerges as a Lagrange multiplier, a feature deemed meaningful from a physical standpoint due to mass conservation.
The solutions under a $L^2$ constraint are commonly referred to as normalized solutions.
 These normalized solutions to \eqref{schrodinger} can be obtained as critical points of the energy functional $J: H^1(\Omega)\to \mathbb{R}$ defined by:
\[J(u):=\frac{1}{2}\int_{\Omega}|\nabla u|^2 dx-\int_{\Omega} F(u) dx\]
on the $L^2$-sphere constraint:
\[\mathcal{S}_c:=\left \{u\in H^1(\Omega):~\int_{\Omega}|u|^2 dx=c \right \},\]
where $F(u)= \int_0^u f(t)dt$.

When $\Omega= \mathbb{R}^N$, problem \eqref{schrodinger}-(\ref{1.2}) is transformed into
\begin{equation}\label{gu}
\begin{cases}
	&-\Delta u+ \lambda u= f(u)  ~~\text{in}~ \mathbb{R}^N,\\
    &\displaystyle\int_{\mathbb{R}^N}|u|^2 dx=c.\\
\end{cases}
\end{equation}
Extensive studies have been conducted in recent years for this problem, particularly when considering nonlinearities satisfying $f(u)\sim|u|^{p-2}u$ as $|u|\to+\infty$, where $p\in (2_*, 2^*)$ with $2_*:=2+\frac{4}{N}$ and $2^*:=\frac{2N}{N-2}$. Within this $L^2$-supercritical range, the corresponding energy functional is unbounded from below on the constraint set $\mathcal{S}_c$. The pioneering work in this direction was carried out by Jeanjean \cite{Jeanjean1997}, where a scaled functional and mountain pass arguments were introduced to address the $L^2$-supercritical problems. Bartsch and Soave \cite{Bartsch2017} developed a natural constraint approach to investigate the $L^2$-supercritical NLS equations and systems on $\mathbb{R}^N$.
For more related results on $\mathbb{R}^N$, we refer to \cite{AJM2022,BFG2023,IkNo,JJLV2022,JL2022,WW-2022} and the associated references.

The exploration of normalized solutions for the NLS on bounded domains was initialed in \cite{NTV-2014}. When $f(u)=|u|^{p-2}u$ with $p$ being $L^2$-supercritical but Sobolev subcritical, Noris et al. \cite{NTV-2014} demonstrated the existence of a positive normalized solution on a unit ball with Dirichlet boundary conditions. The case for general bounded domains was addressed in \cite{PV-2017}. For nonlinear Schr\"odinger systems, one can refer to \cite{NTV-2019}.

 In \cite{PPVV-2021},
Pellacci et al. investigated normalized solutions of the NLS under both Dirichlet and Neumann boundary conditions, focusing on the concentration of solutions at specific points of $\Omega$ as the prescribed mass $c$ varies. Specifically, for the Neumann problem \eqref{schrodinger}--\eqref{1.2}  with $f(u)=u^{p-1}$ for $p\in (2_*, 2^*)$, they employed the Lyapunov-Schmidt reduction method to prove the existence of positive normalized solutions when $c\in (0,c_0)$ for some $c_0>0$. The solutions concentrate at a point $\xi_0\in \overline\Omega$ as $c\to0$, where $\xi_0$ is either a non-degenerate critical point of the mean curvature $H$ of the boundary $\partial\Omega$ or the maximum point of the distance function from $\partial \Omega$. However, the nature of these solutions—whether they correspond to local minimizers or mountain pass-type critical points of the associated functional—remains unclear.

Additionally, it is noteworthy that earlier works on $L^2$-subcritical Schr\"{o}dinger-Poisson type systems under Neumann boundary conditions can be found in \cite{Afonso, PS-2007, PS-2014}, where the authors established the existence of infinitely many normalized solutions using the Ljusternik–Schnirelmann theory.

The methodologies used for bounded domains differ significantly from those utilized in the entire space. In fact, the approaches for $\mathbb{R}^N$ heavily rely on the scaling transform $(t\star u)(x)=t^{\frac{N}{2}}u(tx)$ and the associated  Pohozaev identity. However, due to the lack of invariance under translations and dilation, as well as the emergence of uncontrollable boundary terms in the Pohozaev identity (particularly for nonconvex domains), neither of these techniques is applicable to general bounded domains.

The normalized solutions are also explored within the framework of ergodic Mean Field Games (MFG) systems, offering another key motivation for studying equation \eqref{schrodinger}. Mean-Field Games were introduced in influential works by Huang, Caines, and Malham`e \cite{HMC-2006} and Lasry and Lions \cite{LL-2007}, independently. The primary objective of MFG is to establish a framework for characterizing Nash equilibria in differential games involving an infinite number of agents that are indistinguishable from one another.
For more details, we refer interested readers to \cite{C-2015,CCV-2023, San-2020} and the references therein.
From a mathematical perspective, such equilibria can be characterized by an elliptic system that combines a Kolmogorov equation and a Hamilton-Jacobi-Bellman equation. Moreover, this system has to satisfy normalization in $L^1(\Omega)$ as follows:
\begin{equation}\label{meanfield}
\begin{cases}
	-\Delta v+ H(\nabla v)+ \lambda = h(m(x))  & {\rm in} \,~\Omega ,\\
-\Delta m- div(m \nabla H(\nabla v))=0 & {\rm in} \,~\Omega ,\\
\displaystyle\frac{\partial v}{\partial \nu}=0,~~ \frac{\partial m}{\partial \nu}+ m\nabla H(\nabla v)\cdot \nu=0  & {\rm on} \,~ \partial \Omega , \\
\displaystyle\int_\Omega m dx=1, ~~\int_\Omega v dx=0.
\end{cases}
\end{equation}
The Neumann boundary conditions are based on the assumption that agents' trajectories are restricted to $\Omega$ by bouncing off the boundary in a normal direction.
For the quadratic Hamilton case $H(\nabla v)=|\nabla v|^2$, by using a Hopf-Cole transformation $\phi= e^{-v}/ \int e^{-v}= \sqrt{m}$,
\eqref{meanfield} is reduced to
\begin{equation}\label{singlemeanfield}
\begin{cases}
-\Delta \phi =\lambda \phi- h(\phi^2)\phi  & {\rm in} \,~\Omega,\\
\displaystyle\frac{\partial \phi}{\partial \nu}=0   \, &{\rm on}\,~\partial \Omega,\\
\displaystyle\int_\Omega |\phi|^2 dx=1,
\end{cases}
\end{equation}
which can be viewed as a single nonlinear Schr\"odinger equation with prescribed mass.
In fact, we can check that \eqref{singlemeanfield} is equivalent to \eqref{schrodinger} by taking a simple transformation $\phi= \frac{1}{\sqrt{c}}u$ and $f(u)=h(\frac{u^2}{c})u$.
In \cite{CCV-2023}, Cirant et al. investigated the existence of the viscous ergodic MFG system with Neumann boundary conditions.
 They proved the existence of global minimizers in $L^2$-subcritical and critical cases, as well as local minimizers in $L^2$-supercritical cases.

When $f$ satisfies the $L^2$-supercritical growth, the constraint functional $J|_{\mathcal{S}_c}$  exhibits a mountain pass geometry structure. Hence, it is natural to seek mountain pass type normalized solutions to the Neumann problem (\ref{schrodinger})-(\ref{1.2}). However, in this scenario, the functional $J$ becomes unbounded from below on $\mathcal{S}_c$,  rendering the method in \cite{CCV-2023} ineffective. To the best of our knowledge, there are currently no references addressing this issue. The primary aim of this paper is to develop a novel variational technique to explore the existence of mountain pass type normalized solutions for (\ref{schrodinger})-(\ref{1.2}) under the $L^2$-supercritical growth conditions.

Furthermore, notice that the current existence results of normalized solutions for \eqref{gu} mainly rely on the following global Ambrosetti-Rabinowitz(AR) type condition
\begin{eqnarray}\label{AR}
0<\alpha F(t)\leq f(t)t ~{\rm for}~ t\not=0, ~{\rm where}~ \alpha\in \left(2_*, 2^*\right),
\end{eqnarray}
(see \cite{Bartsch2017,IkNo,Jeanjean1997}), or combined nonlinearities (see  \cite{AJM2022,BFG2023,JJLV2022,JL2022,Soave-2020JDE,Soave-2020JFA,WW-2022}).
Recently, normalized solutions of \eqref{gu} have also been studied when (\ref{AR}) is replaced by certain global monotonicity conditions, as discussed in \cite{BM2020,CLY2023,JL-2020,MS2022}. However, the exploration of normalized solutions on bounded domains with more diverse nonlinearities beyond power nonlinearity is currently limited. Our method is designed to be flexible and applies to a broad range of nonlinearities, including $f(u)=|u|^{p-2}u$, $|u|^{q-2}u+ |u|^{p-2}u$, and $-|u|^{q-2}u+|u|^{p-2}u$, where $2<q<2+\frac{4}{N}<p< 2^*$, in a unified manner. Our results do not require the presence of condition (\ref{AR}) or the global monotonicity conditions. Additionally, motivated by the arguments in \cite{PV-2017}, diverging from the prevalent focus on positive solutions in existing literature, we establish the existence of mountain pass type normalized solutions that are not necessarily positive.

Before stating our results, let us impose the following assumptions on $f$:
\begin{itemize}
	\item[($f_1$)]  $f\in C^1({\mathbb{R}})$, $\lim\limits_{|t|\to 0} \frac{f(t)}{t}=0$;
	 \item[($f_2$)] there exist constants $p\in (2_*, 2^*)$ and $a_0>0$ such that
	\[\lim_{|t|\to \infty} \frac{f(t)}{|t|^{p-2}t} =a_0;\]
	\item[($f_3$)] there exist constants $\mu\geq a_0(p-1)$ and $M>0$ such that
	 \[\mu |t|^{p-2} \leq f'(t), ~\forall |t|\geq M. \]
\end{itemize}

Our main result is the following theorem.
\begin{theorem}\label{mainresult}
	Suppose $(f_1)-(f_3)$ hold. Then there exists a constant $c^*>0$ such that for any $0<c<c^*$, problem (\ref{schrodinger})-(\ref{1.2}) has a normalized solution pair $(u,\lambda)\in H^1(\Omega)\times \mathbb{R}$ of mountain pass type.
\end{theorem}
\begin{remark}
If $f(t)t> 0$ for $t \neq 0$, it allows us to establish the positivity of the Lagrange multiplier $\lambda$. In fact, under this condition, we can obtain a positive normalized solution $u$ of (\ref{schrodinger}). Subsequently, integrating equation (\ref{schrodinger}) and using the Neumann boundary value condition, we can deduce that $\lambda>0$.
\end{remark}

A significant challenge in proving Theorem \ref{mainresult} arises from establishing the boundedness of Palais-Smale sequences. 
To overcome this obstacle, we combine with a parameterized minimax principle with Morse index information for constrained functionals established recently in \cite{BCJS-202210} (see also \cite{BCJS-202212, CJS-2022}) and a new blow-up analysis for the NLS equation under Neumann boundary conditions.
In applying this method, the sign condition $\int_\Omega F(u) \geq 0$ for all $u \in H^1(\Omega)$ is crucial to ensure the applicability of the monotonicity trick (see Theorem \ref{abstract theorem} in Section 3).
However, in our specific problem, $\int_\Omega F(u)$ is not necessarily non-negative when $u$ is not sufficiently large.
Our strategy to tackle this challenge is to utilize a cut-off function and decompose $f$ into two parts, ensuring that one of the parts satisfies this sign condition and allows for the application of the monotonicity trick.

To address the original Neumann problem after deriving solutions for approximation problems, we will propose a innovative blow-up analysis tailored for the NLS subject to Neumann boundary conditions.  In contrast to the blow-up arguments applied in the Dirichlet case (see \cite{EspPet,PV-2017}), the blow-up analysis for NLS with Neumann boundary conditions, subject to the $L^2$-constraint, proves to be considerably more intricate.
In fact, in Neumann problems, the solution is not necessarily required to vanish at the boundary $\partial\Omega$, and local extremum points of solutions can exist on the boundary. This is in stark contrast to Dirichlet problems where such extremum points are found in the interior of $\Omega$.

We note that several studies have conducted blow-up analyses of the NLS under Neumann boundary conditions. 
It is well known that combining blow-up arguments with suitable Liouville-type theorems is highly effective for deriving a priori estimates.
In \cite{LNT-1988},  Lin et al. delivered a comprehensive blow-up analysis to establish the boundedness for positive solutions to the following Neumann problem:
\begin{equation}\label{bingzhan}
\begin{cases}
-d\Delta u +u= g(u)  & {\rm in} \,~\Omega,\\
\displaystyle\frac{\partial u}{\partial \nu}=0   \, &{\rm on}\,~\partial \Omega,\\
\end{cases}
\end{equation}
where $d>0$ and $g$ satisfies Sobolev subcritical growth conditions.
The Sobolev critical case was subsequently addressed by Adimurthi et al. \cite{APY-1993}.
In \cite{HWZ-2016}, Hu et al. used a blow-up argument for a nonlinear Neumann problem involving the $p$-Laplacian to obtain a priori estimates.
For further details on blow-up analyses of elliptic problems under Neumann boundary conditions, the reader is referred to \cite{BGT-2016, LS-2025}.

Blow-up analysis can also establish a connection between the boundedness of solutions and their Morse indices. 
Consider the following problem:
\begin{equation}\label{HARS}
\begin{cases}
-\Delta u = g(x, u)  & {\rm in} \,~\Omega,\\
\displaystyle\frac{\partial u}{\partial \nu}=0   \, &{\rm on}\,~\partial \Omega,\\
\end{cases}
\end{equation}
where $g$ satisfies superlinear and subcritical growth conditions.
Harrabi et al. \cite{HARS-2012} proved that the $L^\infty$ bounds for solutions to \eqref{HARS} are equivalent to bounds on their Morse indices.

In contrast to the arguments presented in \cite{HARS-2012}, our discussion is conducted under a mass constraint and involves a family of nonlinearities arising from varying Lagrange multipliers, along with a dense set used in the monotonicity trick. The most significant distinction, however, is that our blow-up analysis requires an exponential decay estimate for the solution sequence. This difference necessitates substantial modifications to the blow-up arguments in the Neumann case.

  As we discussed earlier, we present a compelling argument to establish that the solution sequence not only exhibits the commonly observed exponential decay away from the boundary $\partial\Omega$, as found in the fixed frequency NLS in the Neumann case (see \cite{NT-1991, NT-1993}), but more significantly, it demonstrates uniform exponential decay away from the blow-up points.

 To conduct the blow-up analysis and subsequently derive a contradiction, a critical step involves demonstrating that the solution sequence exhibits uniform exponential decay, especially when the blow-up points are situated on the boundary $\partial\Omega$. However, in such cases, we can typically only establish that the solution sequence exponentially decays uniformly away from the blow-up points within a smaller domain in $\Omega$, which is distant from the boundary.
Specifically, as detailed in Section 4, through a comparison argument, we can only provide an estimation that the solution sequence exponentially decays uniformly away from the blow-up points in the domain $\tilde{\Omega}_\theta \backslash \cup_{i=1}^k (B_{R\lambda_n^{-\frac{1}{2}}}(P_n^i)\cap \tilde{\Omega}_\theta)$ ($\Omega_\theta$ is defined in Section 4, $\tilde{\Omega}_\theta:= \overline{\Omega}\backslash \Omega_\theta $ and $k\in\{1,2\}$).
To be precise, it is demonstrated that there exist constants $C_1>0$ and $C_2>0$ such that
	\begin{equation*}
		u_n(x)\leq C_1e^{C_1 R}\lambda_n^{\frac{1}{p-2}}\sum_{i=1}^k e^{-C_2\lambda_n^{\frac{1}{2}}|x- P_n^i|},~~\forall x\in \tilde{\Omega}_\theta \backslash \cup_{i=1}^k (B_{R\lambda_n^{-\frac{1}{2}}}(P_n^i)\cap \tilde{\Omega}_\theta),
	\end{equation*}
where $u_n$  satisfies \eqref{withn}, $R>0$ is a constant, $P_n^i$ is a local extremum point, $i\in \{1,\cdots, k\}$.
However, this conclusion falls short of addressing our problem adequately.
 The limitation arises from the challenging nature of estimating the exponential decay of the solution in the neighborhood near the boundary $\partial\Omega$.

 To overcome this limitation, we formulate a diffeomorphism, denoted as $\Phi:\Omega^\theta\to \Omega_\theta$, in a manner that expands the equation to encompass a new domain, namely, $\overline{\Omega}\cup \Omega^\theta$, which combines the original domain $\Omega$ with an adjoining tubular neighborhood $\Omega^\theta$ (see Section 4). Employing the comparison theorem, we deduce that the solution exhibits exponential decay away from the boundary of the newly defined domain $\overline{\Omega}\cup \Omega^\theta$. Subsequently, through a similar argument as mentioned earlier, we observe that the solution exponentially decays uniformly away from the blow-up points in $\overline{\Omega}\backslash \cup_{i=1}^k (B_{R\lambda_n^{-1/2}}(P_n^i)\cap \overline{\Omega})$. This establishes the intended conclusion we strive to attain.




This paper is organized as follows. In Section 2, we provide the mountain pass geometry of the parameterized functionals $J_\rho$ uniformly for $\rho\in[\frac{1}{2},1]$. Subsequently, in Section 3, we show the existence of a mountain pass critical point $u_\rho$ with Morse index information for $J_\rho|_{\mathcal{S}_c}$ for almost every $\rho\in[\frac{1}{2},1]$. In Section 4, we perform a blow-up analysis of solutions with bounded Morse index for the Neumann problem. Finally, in Section 5, we prove Theorem \ref{mainresult}. For convenience, we denote by $\|\cdot\|_r$ the norm of the spaces $L^r(\Omega)(1\le r\le +\infty)$, by $\|\cdot\|$ the norm in $H^1(\Omega)$.

\section{Mountain pass geometry}

This section is dedicated to establishing the uniform mountain pass geometry for a parameterized functional. To achieve this goal, we first decompose $f$ into two parts.
By $(f_1)-(f_2)$, we deduce that there exists $R_0>0$ such that $f(t)t>0$ for $|t|\geq R_0$.
Utilizing this observation, we define
\begin{equation*}
	f_1(t):=\eta (t)f(t)
	~~{\rm and}~~
	f_2(t):=(1-\eta (t))f(t),
\end{equation*}
where $\eta$ is a smooth cut-off function such that
\begin{equation*}
	\eta (t)=
	\begin{cases}
		1, &|t|\geq R_0+1,\\
		0, &|t|< R_0
	\end{cases}
\end{equation*}
with $|\eta'(t)|\leq 2$ for $R_0< |t|< R_0+1$. Clearly, $f_1(t)t\geq 0$ for all $t\not =0$.

Define $J_{\rho}: H^1(\Omega)\to \mathbb{R}$ by
\begin{equation}\label{xiaolu}
	J_\rho(u):=\frac{1}{2}\int_{\Omega}|\nabla u|^2 dx- \int_\Omega F_2(u)dx-\rho\int_{\Omega}F_1(u) dx,~  u\in H^1(\Omega),~ \rho \in \left[\frac{1}{2},1\right],
\end{equation}
where $F_1(u)= \int_0^u f_1(t)dt$ and $F_2(u)= \int_0^u f_2(t)dt$.


We recall the Gagliardo-Nirenberg inequality (see \cite{N-1966}):
for every $N\geq 3$ and $p\in(2,2^*)$, there exists a constant $C_{N,r,\Omega}$ depending on $N$, $r$ and $\Omega$ such that
\[\|u\|_r\leq C_{N,r,\Omega}\|u\|_2^{1-\gamma_r}\|u\|_{H^1(\Omega)}^{\gamma_r},~~\forall u\in H^1(\Omega),\]
where \[\gamma_r:=\frac{N(r-2)}{2r}.\]

Define $\mathcal{B}_\alpha:=\{u\in \mathcal{S}_c: \int_{\Omega} |\nabla u|^2\leq \alpha\}$, $\forall \alpha>0$.
We have the following result.
\begin{lemma}\label{local}
	Assume $(f_1)-(f_3)$. Then there exists $c^*>0$ such that for any $c\in (0, c^*)$, we can find $\alpha^*>0$ such that
	\[\sup_{u\in \mathcal{B}_{\alpha^*}} J_{\frac{1}{2}}(u)< \inf_{u\in \partial\mathcal{B}_{2\alpha^*}} J_1(u).\]
\end{lemma}
\begin{proof}
By $(f_1)-(f_2)$, for any $\epsilon$, $\delta>0$ and $q\in (2, 2_*)$, there exist constants $C'_\epsilon, C_\delta>0$ such that
\begin{eqnarray} 
&&F_2(u)\geq -\frac{\epsilon}{2}|u|^2 -\frac{C'_\epsilon}{q}|u|^q, \label{5-10-1}\\
&&F(u)\leq \frac{\delta}{2}|u|^2 + \frac{C_\delta}{p}|u|^p.\label{5-10-2}
\end{eqnarray}
Let $u\in \mathcal{B}_\alpha$ and $v\in \partial\mathcal{B}_{2\alpha}$, where $\alpha$ is to be determined. By (\ref{5-10-1})-(\ref{5-10-2}), along with the Gagliardo-Nirenberg inequality, we obtain
	\begin{equation*}\label{llll}
		\begin{split}
		& J_1(v)-J_{\frac{1}{2}}(u)\\
		&\geq \frac{1}{2}\int_\Omega |\nabla v|^2-\frac{1}{2}\int_\Omega |\nabla u|^2 -\int_\Omega F(v)+ \int_\Omega F_2(u)\\
		&\geq \frac{1}{2}\int_\Omega \left( |\nabla v|^2- |\nabla u|^2 \right)dx -\frac{\delta+\epsilon}{2}c-\frac{C_\delta C_{N,p}^p}{p}c^{\frac{p(1-\gamma_p)}{2}}\|v\|^{p\gamma_p}- \frac{C'_\epsilon C_{N,q}^q}{q}c^{\frac{q(1-\gamma_q)}{2}}\|u\|^{q\gamma_q}\\
&\geq \frac{\alpha}{2}-\frac{\delta+\epsilon}{2}c-\frac{C_\delta C_{N,p}^p}{p}c^{\frac{p(1-\gamma_p)}{2}}(2\alpha+c)^{\frac{p\gamma_p}{2}}- \frac{C'_\epsilon C_{N,q}^q}{q}c^{\frac{q(1-\gamma_q)}{2}}(\alpha+c)^{\frac{q\gamma_q}{2}}.
\end{split}
\end{equation*}
Due to the arbitrariness of $\epsilon$ and $\delta$, we conclude that for some $c^*>0$ and any $c\in (0, c^*)$, taking $\alpha^*= 4c$, we have
\begin{equation*}
J_1(v)-J_{\frac{1}{2}}(u)\geq c-9^{\frac{p\gamma_p}{2}}\frac{C_1 C_{N,p}^p}{p}c^{\frac{p}{2}}- 5^{\frac{q\gamma_q}{2}}\frac{C_1' C_{N,q}^q}{q}c^{\frac{q}{2}}>0.
\end{equation*}
Hence, the conclusion follows.
\end{proof}

 For any $c\in (0,c^*)$, we define
\[m_\rho:= \inf_{u\in \mathcal{B}_{2\alpha^*}}J_\rho(u), ~\forall\rho\in \left[\frac{1}{2},1 \right].\]
Then, following the approach outlined in \cite[Proposition 3.4]{NTV-2019} (see also \cite[Theorem 1.6]{BJ-2016}), we can derive the following result.
\begin{lemma}\label{exist}
Assume $(f_1)-(f_3)$ and $c\in (0, c^*)$. Then $m_\rho$ is achieved by some $u^*_{\rho}\in \mathcal{B}_{2\alpha^*}\setminus \partial\mathcal{B}_{2\alpha^*}$.
\end{lemma}

\begin{proof}
	Fix $c\in (0, c^*)$. By Lemma \ref{local}, for any $\rho \in \left[\frac{1}{2},1 \right]$, we have
\begin{eqnarray*}\label{Estar}
	\sup_{u\in \mathcal{B}_{\alpha^*}}J_\rho(u)
	 \leq \sup_{u\in \mathcal{B}_{\alpha^*}}J_{\frac{1}{2}}(u)<\inf_{u\in \partial\mathcal{B}_{2\alpha^*}}J_1(u)\leq \inf_{u\in \partial\mathcal{B}_{2\alpha^*}}J_\rho(u).
\end{eqnarray*}
Let $\{u_{\rho,n}\}\subset \mathcal{B}_{2\alpha^*}$ be a minimizing sequence for $J_{\rho}$ at the level $m_{\rho}$. Clearly, $\{u_{\rho,n}\}$ is bounded in $H^1(\Omega)$. Consequently, there exist a subsequence of $\{u_{\rho,n}\}$, still denoted by $\{u_{\rho,n}\}$, and some $u^*_\rho\in H^1(\Omega)$ such that $u_{\rho,n}\rightharpoonup u^*_\rho$ in $H^1(\Omega)$ and $u_{\rho,n}\to u^*_{\rho}$ strongly in $L^r(\Omega)$ for $r\in[1,2^*)$. This implies that 
$u^*_{\rho}\in \mathcal{B}_{2\alpha^*}$ and thus 
$J_\rho(u^*_{\rho})\ge m_{\rho}.$

 By $(f_1)-(f_2)$, together with the H\"older inequality and the Lebesgue dominated convergence theorem, we have
 \begin{eqnarray*}\label{2-19-1}
  \int_{\Omega}\left(F(u_{\rho,n})-F(u^*_{\rho})\right)dx=\int_{\Omega}F(u_{\rho,n}-u^*_{\rho})dx+o_n(1)\to 0.
 \end{eqnarray*}
Therefore, we obtain
 $$
 J_\rho(u^*_{\rho})\le \liminf\limits_{n\to+\infty}J_\rho(u_{\rho,n})=m_{\rho}.
  $$
Combining this with the previous inequality, we deduce that $J_\rho(u^*_{\rho})=m_{\rho}$ and $\|\nabla \left(u_{\rho,n}-u^*_{\rho}\right)\|_2^2\to0$. This implies that $u_{\rho,n}\to u^*_{\rho}$ in $H^1(\Omega)$ and hence $u^*_{\rho}$ is a minimizer of $J_{\rho}$. Thus, the desired conclusion follows.
\end{proof}

Now we can show the mountain pass geometry of $J_{\rho}$ uniformly for $\rho\in[\frac{1}{2}, 1]$.
\begin{lemma}\label{MP}
Assume $(f_1)-(f_3)$. Then, for any $c\in (0,c^*)$,
there exist $w_1, w_2\in \mathcal{S}_c$ such that
\[c_\rho:= \inf_{\gamma\in \Gamma}\max_{t\in[0,1]}J_\rho(\gamma(t))>\max\{J_\rho(w_1), J_\rho(w_2) \} ,~~ \forall \rho\in \left[\frac{1}{2},1 \right], \]
where
	\begin{equation}\label{Gamma}
		\Gamma:=\{\gamma\in C([0,1], \mathcal{S}_c) : \gamma(0)=w_1 , ~\gamma(1)=w_2 \}.
	\end{equation}
\end{lemma}
\begin{proof}

Let $B_r(x)$ denote a ball in $\mathbb{R}^N$, centered at $x\in \mathbb{R}^N$ with radius $r>0$.
Take $\phi \in C^\infty_0(B_1(0))$ with $\phi>0$ in $B_1(0)$ such that $\int_{B_1(0)}\phi^2 =1$.
For $n\in \mathbb{N}$, $x_0\in \Omega$, define
\begin{equation*}
	\varphi_n(x):= c^\frac{1}{2} n^\frac{N}{2} \phi (n(x-x_0)),~x\in \Omega.
\end{equation*}
We can verify that $\varphi_n\in\mathcal{S}_c$ and $supp(\varphi_n)\subset B_\frac{1}{n}(x_0)\subset \Omega$ for sufficiently large $n$.

By $(f_1)-(f_2)$, taking $R_0>0$ larger if necessary, there exist constants  $C_{R_0}, C'_{R_0}, C_p>0$ such that
\begin{eqnarray*}
&&F_1(t)\ge \frac{C_p a_0}{2p}|t|^p,~~ \forall |t|\ge 2R_0,~~~F_1(t)\ge -C_{R_0},~~ \forall |t|\le 2R_0,\\
&&F_2(t)\ge -C'_{R_0}, \forall t\in \mathbb{R}.
\end{eqnarray*}
Set $\Omega_{n, 2R_0} :=\{x\in \Omega: |\varphi_n|\leq 2R_0 \}$.
Then, for sufficiently large $n$ such that $\max\limits_{x\in \overline\Omega}|\varphi_n|> 2R_0$, we have for all $\rho\in [\frac{1}{2},1]$,
\begin{eqnarray*}
J_\rho(\varphi_n)&\leq& \frac{1}{2}\int_{\Omega}|\nabla \varphi_n(x)|^2 + (C_{R_0}+C'_{R_0})|\Omega| -\rho\int_{\Omega\backslash \Omega_{n, 2R_0}}  \frac{C_p a_0}{2p} |\varphi_n|^p \\
&=& \frac{cn^2}{2}\int_{B_1(0)}|\nabla \phi(x)|^2 +(C_{R_0}+C'_{R_0})|\Omega|+ \frac{2^pC_P a_0}{4p}R_0^p|\Omega|\\
&&-\frac{C_Pa_0}{4p}c^\frac{p}{2} n^{\frac{pN-2N}{2}}\int_{B_1(0)}|\phi(x)|^p\to -\infty.
\end{eqnarray*}
Hence, there exists $n_0>0 $ sufficiently large such that
\begin{equation*}
	J_\rho(\varphi_{n_0})<J_1(u^*_{\frac{1}{2}})\le J_{\rho}(u^*_{\frac{1}{2}}),  ~\forall \rho\in \left[\frac{1}{2},1 \right].
\end{equation*}
Choose $w_1=u^*_{\frac{1}{2}}$ and $w_2=\varphi_{n_0}$. Clearly, $u^*_{\frac{1}{2}}\in \mathcal{B}_{2\alpha^*}\setminus \partial\mathcal{B}_{2\alpha^*}, \varphi_{n_0}\not\in \mathcal{B}_{2\alpha^*}$.
By continuity,
for any $\gamma \in \Gamma$, there exists $t_\gamma \in [0,1]$ such that $ \gamma(t_\gamma)\in \partial \mathcal{B}_{2\alpha^*}$. Thus, by Lemma \ref{local}, it follows that
\begin{align*}
	\max_{t\in[0,1]}J_\rho(\gamma(t))\geq J_\rho(\gamma(t_\gamma)) &\geq  \inf_{u\in \partial \mathcal{B}_{2\alpha^*}}J_\rho(u)>\sup_{u\in \mathcal{B}_{\alpha^*}}J_{\frac{1}{2}}(u)\\
	& \geq J_\rho(u^*_{\frac{1}{2}})=\max\{J_\rho(w_1), J_\rho(w_2)\},  ~\forall \rho\in \left[\frac{1}{2},1 \right].
\end{align*}
The proof is now complete.
\end{proof}

\section{Existence of MP solutions for a dense set}

In this section, we establish the existence of a bounded Palais-Smale sequence at level $c_\rho$ for almost every $\rho\in [\frac{1}{2},1]$.
 Our approach involves applying a recently developed min-max principle on the $L^2$-sphere, as detailed in \cite{BCJS-202210}. This principle integrates the monotonicity trick presented in \cite{J-PRSE1999} with the min-max theorem enriched by second-order insights from Fang and Ghoussoub \cite{FG-1994}, which is also elaborated upon in \cite[Chapter 11]{Ghoussoub}.

For a domain $D\subset \mathbb{R}^N$ and $\phi, u\in H^1(D)$, we consider
\begin{equation*}\label{morse}
	Q_{\lambda,\rho}(\phi;u;D):= \int_{D}|\nabla \phi |^2 dx+ \lambda\int_{D}|\phi |^2 dx -\int_{D} f'_2(u)\phi^2 dx -\int_{D}\rho f'_1(u)\phi^2 dx,
\end{equation*}
where $\lambda\in \mathbb{R}, \rho\in[\frac{1}{2},1]$.
The Morse index of $u$, denote by $m(u)$, is the maximum dimension of a subspace $W\subset H^1(D)$ such that $Q_{\lambda, \rho}(\phi;u;D)<0$ for all $\phi \in W\backslash \{0\}$.


To state the abstract minimax theorem, we recall a general setting introduced in \cite{BL-1983}.
Let $(E, \langle\cdot,\cdot\rangle)$ and $(H,(\cdot, \cdot))$
be two infinite dimensional Hilbert spaces such that
\[ E\hookrightarrow H \hookrightarrow E'\]
with continuous injections. The continuous injection $E\hookrightarrow H$  has a norm at most 1 and $E$ is identified with its image in $H$. For $u\in E$, we denote $\|u \|^2= \langle u,u \rangle$ and $|u|^2= (u,u)$. For $a\in (0, +\infty)$, we define $S(a):= \{u\in E, |u|^2= a\}$. We denote by $\|\cdot\|_*$ and $\|\cdot\|_{**}$, respectively, the operator norm of $\mathcal{L}(E, R)$ and of $\mathcal{L}(E, \mathcal{L}(E, \mathbb{R}))$.
\begin{definition} \cite{BCJS-202210}
	Let $\phi: E\to \mathbb{R}$ be a $C^2$-functional on $E$ and $\alpha\in (0,1]$. We say that $\phi'$ and $\phi''$ are $\alpha$-H\"{o}lder continuous on bounded sets if for any $R>0$ one can find $M=M(R)>0$ such that for any $u_1, u_2\in B(0, R)$:
	\begin{equation}
		\|\phi'(u_1)-\phi'(u_2)\|_* \leq M\|u_1-u_2\|^\alpha,~~
		\|\phi''(u_1)-\phi''(u_2)\|_{**} \leq M\|u_1-u_2\|^\alpha.
	\end{equation}
\end{definition}
\begin{definition} \cite{BCJS-202210}
	Let $\phi$ be a $C^2$-functional on $E$, for any $u\in E$ define the continuous bilinear map:
	\[D^2\phi(u):= \phi''(u)-\frac{\phi'(u)\cdot u}{|u|^2}(\cdot,	\cdot). \]
\end{definition}
\begin{definition}\cite{BCJS-202210}\label{amor}
For any $u\in S(a)$ and $\theta>0$, we define an approximate Morse index by
\begin{align*}
\tilde{m}_\theta(u):=\sup \{dim ~L | &L~ {\rm is~ a~ subspace ~of~ } T_uS(a)~ {\rm such~ that~} D^2\phi(u)[\phi,\phi]<-\theta\|\phi\|^2, ~\forall \varphi \in L \backslash \{0\}\}.
\end{align*}
\end{definition}
If $u$ is a critical point for the constrained functional $\phi|_{S(a)}$ and $\theta=0$, then $\tilde{m}_\theta(u)$ is the Morse index of $u$ as a constrained critical point.

\begin{theorem}\label{abstract theorem}
	(\cite[Theorem 1.5]{BCJS-202210}).
	Let $I\subset(0, +\infty)$ be an interval and consider a family of $C^2$ functionals $\Phi_\rho : E \to \mathbb{R}$ of the form:
	\begin{equation*}
		\Phi_\rho(u)= A(u)-\rho B(u), ~~\rho \in I,
	\end{equation*}
where $B(u)\geq 0$ for every $u\in E$, and
\begin{equation*}
	{\rm either~~ }A(u)\to +\infty ~~{\rm or}~~ B(u)\to +\infty
	{\rm ~~as~~} u\in E ~~{\rm and ~~}\|u\|\to +\infty.
\end{equation*}	
 Suppose moreover that $\Phi'_\rho$ and $\Phi''_\rho$ are $\alpha$-H\"{o}lder continuous on bounded sets for some $\alpha \in (0,1]$.
Finally, suppose that there exist $w_1, w_2\in S(a)$ (independent of $\rho$) such that,  set 
\begin{equation*}
	\Gamma:=\{\gamma\in C([0,1], S(a)): \gamma(0)=w_1, \gamma(1)=w_2\},
\end{equation*}
we have
\begin{equation}
	c_\rho:=\inf_{\gamma\in \Gamma}\max_{t\in [0,1]}\Phi_\rho(\gamma(t))
	>\max\{\Phi_\rho(w_1), \Phi_\rho(w_2)\},~~\forall \rho \in I.
\end{equation}
Then, for almost every $\rho\in I$, there exist	sequence $\{u_n\}\subset S(a)$ and $\zeta_n\to 0^+$ such that, as $n \to \infty$,
\begin{itemize}
	\item[(i)]  $\Phi_\rho(u_n)\to c_\rho$;
	\item[(ii)] $\|\Phi'_\rho|_{S(a)}(u_n)\|\to 0$;
	\item[(iii)]  $\{u_n\}$ is bounded in $E$;
	\item[(iv)] $\tilde{m}_{\zeta_n}(u)\leq 1$.
\end{itemize}
\end{theorem}

Define $h_{\rho}(t)=\rho f_1(t)+ f_2(t)$ for $t\in \mathbb{R}$ and $\rho\in [\frac{1}{2},1]$.
In the following, we obtain the main result of this section.

\begin{theorem}\label{Thm-appro}
	Assume $(f_1)-(f_3)$ and $c\in (0, c^*)$. Then, for almost every $\rho\in [\frac{1}{2}, 1]$, there exists a critical point $u_\rho$ of $J_\rho$ on $\mathcal{S}_c$ at level $c_\rho$, which solves the following problem
	
\begin{equation}\label{approximate}
\begin{cases}
-\Delta u_\rho+ \lambda_\rho u_\rho=h_\rho(u_\rho)  & {\rm in} \,~\Omega,\\
\frac{\partial u_\rho}{\partial \nu}=0  \, &{\rm on}\,~\partial \Omega
\end{cases}
\end{equation}
for some $\lambda_\rho \in \mathbb{R}$.
Moreover, the Morse index of $\{u_{\rho}\}$ satisfies $m(u_\rho)\le 2$.
\end{theorem}
\begin{proof}
We will apply Theorem \ref{abstract theorem} to the family of functionals $J_\rho$, where $E=H^1(\Omega)$, $H=L^2(\Omega)$, $S(a)=\mathcal{S}_c$ and $\Gamma$ is defined by \eqref{Gamma}.
Specifically, we set
\begin{equation*}
	A(u)= \frac{1}{2}\int_{\Omega}|\nabla u|^2dx- \int_\Omega F_2(u) dx~~{\rm and }~~B(u)=\int_{\Omega} F_1(u) dx.
\end{equation*}
Thus, we have $J_\rho(u)= A(u)-\rho B(u).$
Given that $u\in \mathcal{S}_c$ and considering the boundedness of $\int_\Omega F_2(u)dx$, we deduce that
\[A(u)\to \infty~~{\rm as }~~ \|u\|\to +\infty.\]
Moreover, by assumptions $(f_1)-(f_2)$, it follows that $J'_\rho$ and $J''_\rho$ are locally H\"{o}lder continuous on $\mathcal{S}_c$.
By Lemma \ref{MP}, we can apply Theorem \ref*{abstract theorem} to produce a bounded Palais-Smale sequence $\{u_n\}\subset H^1(\Omega)$ for the constrained functional $J_{\rho}|_{\mathcal{S}_c}$ at level $c_\rho$ for almost every $\rho\in [\frac{1}{2}, 1]$.
 Additionally, there exists a sequence $\zeta_n \to 0^+$ such that $\tilde{m}_{\zeta_n}(u_n)\leq 1$.
 
Since $\|J'_\rho |_{\mathcal{S}_c}(u_n) \|\to 0$, and by the boundedness of $\{u_n\}$, there exists a sequence $\{\lambda_n\}\subset \mathbb{R}$ such that for any $\varphi\in H^1(\Omega)$, we have
\begin{equation}\label{weaksolutionsequence}
\int_{\Omega} \nabla u_n \nabla \varphi dx+\lambda_n \int_{\Omega} u_n \varphi dx - \int_{\Omega} h_{\rho}(u_n) \varphi dx =o(1).
\end{equation}
This implies that
\begin{equation*}\label{unastest}
	\int_{\Omega}|\nabla u_n|^2 dx + \lambda_n c-\int_{\Omega} h_{\rho}(u_n)u_n dx\to 0.
\end{equation*}
Using $(f_1)-(f_2)$ again, we deduce that $\{\lambda_n\}$ is bounded. Therefore, up to a subsequence, we may assume that $\lambda_n \to \lambda_\rho \in \mathbb{R}$ and $u_n\rightharpoonup u_{\rho}$ weakly in $H^1(\Omega)$.
By \eqref{weaksolutionsequence}, we obtain
\begin{equation*}
	\int_{\Omega} \nabla u_\rho \nabla \varphi dx+\lambda_\rho\int_{\Omega} u_\rho \varphi dx -\int_{\Omega}h_{\rho}(u_\rho) \varphi dx=0,
\end{equation*}
which implies that $u_\rho$ weakly solves \eqref{approximate}. By the compact embedding $H^1(\Omega)\hookrightarrow L^r(\Omega)$ for $r\in[1,2^*)$ and standard arguments, we obtain that $u_n\to u_{\rho}$ strongly in $H^1(\Omega)$.

It remains to show that $m(u_\rho)\le 2$.
Since $T_u\mathcal{S}_c$ has codimension 1, noting that $d^2|_{\mathcal{S}_c}J_\rho$ and $T_{u_n}\mathcal{S}_c$ vary with continuity, by $\tilde{m}_{\zeta_n}(u_n)\leq 1$ it follows that $\tilde{m}_0(u_\rho) \leq 1$. 
Then, we can use similar arguments as in \cite[Proposition 3.5]{BCJS-202212} to show $m(u_\rho)\leq 2$.
In fact,  since the tangent space $T_{u}\mathcal{S}_c$ has codimension $1$, it suffices to show that $u_\rho\in \mathcal{S}_c$ has Morse index at most $1$ as a constrained critical point. If this were not the case, by Definition \ref{amor}, we may assume by contradiction that there exists a subspace $W_0 \subset T_u\mathcal{S}_c$ with $dim W_0 = 2$ such that
\begin{equation}\label{bumiao}
	D^2J_\rho(u_\rho)[w, w]<0~\mbox{for~all}~w\in W_0\backslash \{0\}.
\end{equation}
Since $W_0$ is finite-dimensional, there exists a constant $\beta>0$ such that
\begin{equation*}
	D^2J_\rho(u_\rho)[w, w]<-\beta~\mbox{for~all}~w\in W_0\backslash \{0\}~~\mbox{with}~~\|w\|=1.
\end{equation*}
Using the homogeneity of $D^2J_\rho(u_\rho)$, we deduce that 
\begin{equation*}
	D^2J_\rho(u_\rho)[w, w]<-\beta \|w\|^2~\mbox{for~all}~w\in W_0\backslash \{0\}.
\end{equation*}
Now,  since $J_\rho'$ and $J_\rho''$ are $\alpha$-H\"{o}lder continuous on bounded sets for some $\alpha\in (0,1]$, it follows that there exists a sufficient small $\delta_1>0$ such that, for any $v\in \mathcal{S}_c$ satisfying $\|v-u_\rho\|\leq \delta_1$,
\begin{equation}\label{sanbai}
	D^2J_\rho(v)[w, w]<-\frac{\beta}{2} \|w\|^2~\mbox{for~all}~w\in W_0\backslash \{0\}.
\end{equation}
Hence, using the fact that $\|u_n-u_\rho\| \leq\delta_1$ for sufficiently large $n\in \mathbb{N}$,  and in view of \eqref{bumiao}, \eqref{sanbai}, and $\zeta_n\to 0^+$, we obtain
\begin{equation*}\label{qifur}
	D^2J_\rho(u_n)[w, w]<-\frac{\beta}{2} \|w\|^2<\zeta_n \|w\|^2~\mbox{for~all}~w\in W_0\backslash \{0\}
\end{equation*}
for any such large $n$. Since dim$W_0>1$, this provides a contradiction with Theorem \ref{abstract theorem} (iv), recalling that $\zeta_n\to 0$. 

\end{proof}

\begin{remark}
Note that for any $\rho\in [\frac{1}{2}, 1]$, the constant function $u_c:=\left(\frac{c}{|\Omega|}\right)^{\frac{1}{2}}$ is always a solution of (\ref{approximate}) on $\mathcal{S}_c$ for $\lambda=\frac{h_{\rho}(u_c)}{u_c}$. 
Under the assumptions $(f_1)-(f_2)$, we  can compute the constraint Morse index  $\tilde{m}_0(u_c)$ corresponding to $J_{\rho}$, as demonstrated in \cite[Proposition 2.1]{CJS-2022} (see also \cite[Proposition 4.1]{CDS}).  Specifically, for any $c\in(0,c^*)$ with some proper $c^*>0$, we have $\tilde{m}_0(u_c)=0$ for all $\rho\in[\frac{1}{2}, 1]$. This implies that $u_c$ is a local minimizer of $J_{\rho}$ for every $\rho\in [\frac{1}{2}, 1]$. Based on this observation, we may select $w_1=u_c$ in Lemma 2.3 to construct the uniform mountain pass geometry for $J_{\rho}$. Furthermore, the mountain pass type solution $u$ obtained in Theorem 1.1 cannot be a constant function, as this would contradict the fact that 
$u_c$ is a local minimizer.
\end{remark}

\section{Blow up analysis}\label{blowup}
In this section, we develop a blow-up analysis for the sequence $\{u_{\rho_n}\}$. The goal of this analysis is to prove that $\{u_{\rho_n}\}$ is bounded in $H^1(\Omega)$. Consequently, we aim to show that $\{u_{\rho_n}\}$ converges strongly in $H^1(\Omega)$ to a constrained critical point of $J_1$ as $\rho_n\to 1^-$.

For simplicity, we
denote $u_n:= u_{\rho_n}$, $\lambda_n:=\lambda_{\rho_n}$, $c_n:= c_{\rho_n}$ in the following discussion. Here, $u_n$ weakly solves the following problem
\begin{equation}\label{withn}
\begin{cases}
-\Delta u_n+ \lambda_n u_n=h_{\rho_n}(u_n)  & {\rm in} \,~ \Omega,\\
\frac{\partial u_n}{\partial \nu}=0 \, &{\rm on}\,~\partial \Omega,\\
\int_{\Omega}|u_n|^2 dx=c,
\end{cases}
\end{equation}
where $\lambda_n\in \mathbb{R}$ and $\rho_n \to 1^-$. By Theorem \ref{Thm-appro}, we have $m(u_n)\le2$. Using standard regularity arguments, we obtain $u_n\in C^2(\overline{\Omega})$.

\begin{lemma}\label{boundedbelow}
	There exists a constant $C$ such that $\lambda_n\geq C$ for all $n$.
\end{lemma}
\begin{proof}
	We assume by contradiction that $\lambda_n\to -\infty$.
	Let $V$ be a subspace of $H^1(\Omega)$ with dimension $k$, where $k>2$.
	Define $\Omega_{M_0}:=\{x\in \Omega: |u_n(x)|\leq M_0\}$, where $M_0=\max\{M, R_0+1\}$ and $R_0$ is given in Section 2.
By assumptions $(f_1)$ and $(f_3)$, there exist constants $C_{M_0}, C'_{M_0}>0$ such that for any $\phi\in H^1(\Omega)$,
\begin{equation*}
	\begin{split}
		\int_\Omega h_{\rho_n}'(u_n)\phi^2 dx &=
		\int_{\Omega_{M_0}} h_{\rho_n}'(u_n)\phi^2 dx+\int_{\Omega\backslash \Omega_{M_0}}h_{\rho_n}'(u_n)\phi^2 dx\\
		&\geq \int_{\Omega_{M_0}} -C_{M_0}\phi^2 dx+ \int_{\Omega\backslash \Omega_{M_0}}h_{\rho_n}'(u_n)\phi^2 dx\\
		&\geq \int_{\Omega_{M_0}} -C_{M_0}\phi^2 dx+ \int_{\Omega\backslash \Omega_{M_0}}\left(\rho_n\mu |u_n|^{p-2}- C_{M_0}'\right)\phi^2 dx.
	\end{split}
\end{equation*}
	Taking $\varphi\in V\setminus\{0\}$, we obtain
	\begin{equation*}\label{jiaoxue}
		\begin{split}
		Q_{\lambda_n,\rho_n}(\varphi;u_n;\Omega)
		&\leq \int_{\Omega}|\nabla \varphi |^2 dx+ \lambda_n\int_{\Omega}\phi ^2 dx  +\int_{\Omega_{M_0}} C_{M_0} \varphi^2 dx\\
		&-\int_{\Omega\backslash \Omega_{M_0}} \left(\mu |u_n|^{p-2}- C_{M_0}'\right)\varphi^2 dx\\
		&\leq \|\varphi\|^2 + \left(\lambda_n+C_{M_0}+ C_{M_0}'-1\right)\int_{\Omega}\varphi^2 dx.
		\end{split}
	\end{equation*}
This implies that $Q_{\lambda_n,\rho_n}(\varphi;u_n;\Omega)$ is negative definite on $V$ for sufficiently large $n$, which contradicts the fact that $m(u_n)\le 2$.
\end{proof}

\begin{lemma}
If $\lambda_n \to +\infty$, then $\|u_n \|_{L^\infty} \to +\infty$.
\end{lemma}
\begin{proof}
By \eqref{withn} and assumptions $(f_1)-(f_2)$,  there exists a constant $C_1>0$ such that
 \begin{equation*}
 	\lambda_nc\leq \int_\Omega f(u_n)u_n \leq c+\int_\Omega C_1 |u_n|^p dx\le c+C_1|\Omega|\|u_n \|_{L^{\infty }}^p.
 \end{equation*}
This implies that
 \begin{equation*}
 	\|u_n \|_{L^{\infty }}\geq \left(\frac{c}{C_1|\Omega |} \right)^{\frac{1}{p}} (\lambda_n-1)^{\frac{1}{p}} \to +\infty.
 \end{equation*}
\end{proof}

In the following, we will analyze the asymptotic behavior of the solution to \eqref{withn} as $\lambda_n \to +\infty$. For simplicity, we may assume without loss of generality that $\max\limits_{x\in\overline{\Omega}}u_n(x)>0$.

We begin by providing a local description of the blow-up points.
\begin{lemma}\label{localbe}
 Suppose that $\lambda_n \to +\infty$.
 Let $P_n\in \overline{\Omega}$ be such that, for some $R_n\to \infty$,
 \begin{equation*}
 	|u_n(P_n)|=\max_{B_{R_n\tilde{\epsilon}_n}(P_n) \cap \overline{\Omega}}|u_n(x)| ~~where~~
 	\tilde{\epsilon}_n=a_0^{-\frac{1}{2}}|u_n(P_n)|^{-\frac{p-2}{2}} \to 0.
 \end{equation*}
 Set $\epsilon_n=\lambda_n^{-\frac{1}{2}}$. Then
  \begin{equation}\label{3.8.1}
 \left(\frac{\tilde{\epsilon}_n}{\epsilon_n} \right)^2\to \tilde{\lambda} \in (0, a_0].
 \end{equation}
  Suppose moreover that
 \begin{equation}\label{inner}
 	\limsup_{n\to +\infty}\frac{dist(P_n, \partial\Omega)}{\tilde{\epsilon}_n}= +\infty.
 \end{equation}
 Then, passing to a subsequence if necessary, we have
 \begin{itemize}
    \item[(i)]  $P_n\to P\in \Omega$;
 	\item[(ii)] 
 	$\frac{dist(P_n, \partial\Omega)}{\epsilon_n}\to +\infty ~~as ~~n\to +\infty,$
 	and the scaled sequence
 	\begin{equation}\label{equivalent}
 		v_n(x):=a_0^{\frac{1}{p-2}} \epsilon_n^{\frac{2}{p-2}}u_n(\epsilon_n x +P_n) ~~{\rm for}~~ x\in \Omega_n:=\frac{\Omega-P_n}{\epsilon_n}
 	\end{equation}
 	converges to some $v\in H^1(\mathbb{R}^N)$ in $C^2_{loc}(\mathbb{R}^N)$, where $v$ satisfies
 \begin{equation}
\left\{\begin{array}{ll}
-\Delta v+ v=|v|^{p-2}v  & {\rm in} \,~ \mathbb{R}^N,\\
|v(0)|=\max\limits_{x\in \mathbb{R}^N} v,  \\
v(x)\to 0  ~~{\rm as}~~|x|\to +\infty;
\end{array}\right.
\end{equation}
\item[(iii)] there exists $\phi_n\in C_0^{\infty}(\Omega)$, with supp$\phi_n \subset B_{R\epsilon_n}(P_n)$ for some $R>0$, such that
	$Q_{\lambda_n,\rho_n}(\phi_n;u_n;\Omega)<0$;
	\item[(iv)]  for all $R>0$ and $q\geq 1$, 
	\[	\lim_{n\to \infty}\lambda_n^{\frac{N-2}{2}-\frac{q}{p-2}}\int_{B_{R\epsilon_n}(P_n)}|u_n|^q dx=\lim_{n\to \infty}\int_{B_{R}(0)}|v_n|^q dy =\int_{B_{R}(0)}|v|^q dy.
	\]
 \end{itemize}
 Instead of \eqref{inner}, we suppose that
  \begin{equation}\label{onboundary}
 	\limsup_{n\to +\infty}\frac{dist(P_n, \partial\Omega)}{\tilde{\epsilon}_n}< +\infty.
 \end{equation}
 Then, passing to a subsequence if necessary, the following results hold:
 \begin{itemize}
	\item[(i)]  $P_n\to P\in \partial\Omega$;
	\item[(ii)]
 $		\frac{dist(P_n, \partial\Omega)}{\epsilon_n}\to d_0\geq 0 ~~as ~~n\to +\infty$,
 	and the scaled sequence $\{v_n\}$ defined in (\ref{equivalent}) converges to some $v\in H^1(\mathbb{R}^N_+)$ in $C^2_{loc}(\mathbb{R}^N_+)$ as $n\to \infty$, where $v$ satisfies
 \begin{equation}
\left\{\begin{array}{ll}
-\Delta v+ v=|v|^{p-2}v  & {\rm in} \,~ \mathbb{R}^N_+,\\
\displaystyle\frac{\partial v}{\partial x_N}=0 \, &{\rm on}\,~\partial \mathbb{R}^N_+,\\
|v(0)|=\max\limits_{x\in \mathbb{R}^N_+} v,  \\
v(x)\to 0  ~~{\rm as}~~|x|\to +\infty;
\end{array}\right.
\end{equation}
	\item[(iii)] there exists $\phi_n\in C_0^{\infty}(\Omega)$, with supp$\phi_n \subset B_{R\epsilon_n}(P_n)\cap \overline{\Omega}$ for some $R>0$, such that
	$Q_{\lambda_n, \rho_n}(\phi_n;u_n;\Omega)<0$;
	
	\item[(iv)] for all $R>0$ and $q\geq 1$,
	\begin{equation*}
	\lim_{n\to \infty}\lambda_n^{\frac{N-2}{2}-\frac{q}{p-2}}\int_{B_{R\epsilon_n}(P_n)\cap \Omega}|u_n|^q dx=\lim_{n\to \infty}\int_{B_{R}(0)\cap \Omega_n}|v_n|^q dy =\int_{B_{R}(0)\cap \mathbb{R}^N_+}|v|^q dy.
	\end{equation*}
	
\end{itemize}
\end{lemma}

\begin{proof}
Since $u_n$ may change sign, $P_n$ can be either a positive local maximum or a negative local minimum point.
For simplicity, we focus on the case where $P_n$ is a positive local maximum point; the arguments for the negative local minimum case are analogous.

By (\ref{withn}), we get
\begin{equation*}
		0\leq \frac{-\Delta u_n(P_n)}{u_n(P_n)}= \frac{f_2(u_n(P_n))+\rho_n f_1(u_n(P_n))}{u_n(P_n)}- \lambda_n.
\end{equation*}
Using $(f_2)$ and Lemma \ref{boundedbelow}, we deduce that
\[\frac{\lambda_n}{|u_n(P_n)|^{p-2}}\to \tilde{\lambda}\in [0, a_0] ~{\rm as}~n\to \infty. \]

Next, we show that  $\tilde{\lambda}>0$.

Define the rescaled function
\begin{equation*}
	\tilde{u}_n(x):=a_0^{\frac{1}{p-2}} \tilde{\epsilon}_n^{\frac{2}{p-2}}u_n(\tilde{\epsilon}_nx+ P_n)~~{\rm for}~~ x\in \tilde{\Omega}_n:=\frac{\Omega-P_n}{\tilde{\epsilon}_n}.
\end{equation*}
Clearly, $\tilde{u}_n$ satisfies
\begin{equation}\label{chua}
	\left\{\begin{array}{ll}
		-\Delta \tilde{u}_n +\lambda_n \tilde{\epsilon}_n^2 \tilde{u}_n =a_0^{\frac{1}{p-2}}\tilde{\epsilon}_n^{\frac{2p-2}{p-2}}h_{\rho_n}(a_0^{-\frac{1}{p-2}}\tilde{\epsilon}_n^{-\frac{2}{p-2}}\tilde{u}_n)  & {\rm in} \,~ \tilde{\Omega}_n,\\
	|\tilde{u}_n(x)|\leq |\tilde{u}_n(0)|=1  & {\rm in} \,~ \tilde{\Omega}_n, \\
\displaystyle	\frac{\partial \tilde{u}_n}{\partial \nu}=0 \, &{\rm on}\,~\partial\tilde{\Omega}_n.
	\end{array}\right.
\end{equation}
Let $d_n=dist(P_n, \partial \Omega)$. Then
\begin{equation*}
	\frac{d_n}{\tilde{\epsilon}_n}:= L\in [0, +\infty]
	~~{\rm and}~~\tilde{\Omega}_n\to
	\begin{cases}
	&\mathbb{R}^N ~~{\rm if}~~L=+\infty;\\
	&\mathbb{H}~~~~{\rm if}~~L<+\infty,	
	\end{cases}
\end{equation*}
where $\mathbb{H}$ denotes a half-space such that $0\in \overline{\mathbb{H}}$ and $d(0, \partial \mathbb{H})=L$.
By regularity arguments, up to a subsequence, $\tilde{u}_n\to \tilde{u}$ in $C^2_{loc}(\overline{D})$, where $\tilde{u}$ solves
\begin{equation}\label{tran1}
\left\{\begin{array}{ll}
-\Delta \tilde{u}+ \tilde{\lambda} \tilde{u}=|\tilde{u}|^{p-2}\tilde{u}  & {\rm in} \,~ D,\\
|\tilde{u}(x)|\leq |\tilde{u}(0)|=1  & {\rm in} \,~ D, \\
\displaystyle\frac{\partial \tilde{u}}{\partial \nu}=0 \, &{\rm on}\,~\partial D,
\end{array}\right.
\end{equation}
where $D$ is either $\mathbb{R}^N$ or $\mathbb{H}$.

We claim that $m(\tilde{u})\le 2$. To see this, suppose for contradiction that there exists
$k>2$ such that there are $k$ positive functions $\phi_1, \cdots, \phi_k \in H^1(D)$, orthogonal in $L^2(\Omega)$, satisfying
 \begin{equation}\label{2-20-1}
 \int_D |\nabla \phi_i|^2 dx + \tilde{\lambda}\int_D \phi^2 dx- \int_D (p-1)|\tilde{u}|^{p-2}\phi^2 dx<0
\end{equation}
for every $i\in \{1,\cdots, k\}$.

Define the rescaled functions
\[\phi_{i,n}(x):= \tilde{\epsilon}_n^{-\frac{N-2}{2}}\phi_i \left(\frac{x-P_n}{\tilde{\epsilon}_n}\right).\]
Additionally, let
\[\tilde{\Omega}_{n, M_0}:= \{x\in \tilde{\Omega}_n: |a_0^{-\frac{1}{p-2}}\tilde{\epsilon}_n^{-\frac{2}{p-2}}\tilde{u}_n|\leq M_0\}, ~~\tilde{\Omega}_{n, M_0}^c:= \tilde{\Omega}_n\backslash \tilde{\Omega}_{n, M_0},\] where $M_0$ is as given in Lemma \ref{boundedbelow}.

By direct computations, there exist constants $\tilde{C}_{M_0}, \tilde{C}'_{M_0}>0$ such that the following estimates hold:
\begin{align*}
			\int_{\Omega} f'_1(u_n) \phi_{i, n}^2 dx &= \int_{\tilde{\Omega}_n} \tilde{\epsilon}^2_n \phi_i^2 f_1'(a_0^{-\frac{1}{p-2}}\tilde{\epsilon}_n^{-\frac{2}{p-2}} \tilde{u}_n) dx\\
		&=\int_{\tilde{\Omega}_{n, M_0+1}} [\tilde{\epsilon}^2_n\phi_i^2 f'(a_0^{-\frac{1}{p-2}}\tilde{\epsilon}_n^{-\frac{2}{p-2}}\tilde{u}_n)\eta(a_0^{-\frac{1}{p-2}}\tilde{\epsilon}_n^{-\frac{2}{p-2}}\tilde{u}_n)\\
		&+ \tilde{\epsilon}^2_n \phi_i^2 f(a_0^{-\frac{1}{p-2}}\tilde{\epsilon}_n^{-\frac{2}{p-2}}\tilde{u}_n)\eta'(a_0^{-\frac{1}{p-2}}\tilde{\epsilon}_n^{-\frac{2}{p-2}}\tilde{u}_n) ]dx  \\
		&+\int_{\tilde{\Omega}_{n, M_0}^c} \tilde{\epsilon}^2_n \phi_i^2 f'(a_0^{-\frac{1}{p-2}}\tilde{\epsilon}_n^{-\frac{2}{p-2}}\tilde{u}_n)\eta(a_0^{-\frac{1}{p-2}}\tilde{\epsilon}_n^{-\frac{2}{p-2}}\tilde{u}_n) dx\\
		&\geq \int_{\tilde{\Omega}_{n, M_0} \cap supp\{\phi_i\}} -\tilde{C}_{M_0} \tilde{\epsilon}^2_n \phi_i^2 dx\\
		&+ \int_{\tilde{\Omega}_{n, M_0}^c} \tilde{\epsilon}^2_n \phi_i^2\mu |a_0^{-\frac{1}{p-2}}\tilde{\epsilon}_n^{-\frac{2}{p-2}}\tilde{u}_n|^{p-2} dx\\
		&= \int_{\tilde{\Omega}_{n, M_0} \cap supp\{\phi_i\}} -\tilde{C}_{M_0} \tilde{\epsilon}^2_n \phi_i^2 dx + \int_{\tilde{\Omega}_n} \tilde{\epsilon}^2_n \phi_i^2\mu |a_0^{-\frac{1}{p-2}}\tilde{\epsilon}_n^{-\frac{2}{p-2}}\tilde{u}_n|^{p-2} dx\\
		&-\int_{\tilde{\Omega}_{n, M_0}} \tilde{\epsilon}^2_n \phi_i^2\mu |a_0^{-\frac{1}{p-2}}\tilde{\epsilon}_n^{-\frac{2}{p-2}}\tilde{u}_n|^{p-2} dx
\end{align*}
and
\begin{align*}
	\int_{\Omega} f'_2(u_n) \phi_{i, n}^2 dx &= \int_{\tilde{\Omega}_n} \tilde{\epsilon}^2_n\phi_i^2 f_2'(a_0^{-\frac{1}{p-2}}\tilde{\epsilon}_n^{-\frac{2}{p-2}}\tilde{u}_n) dx\\	&=\int_{\tilde{\Omega}_{n, M_0}\cap supp\{\phi_i\}} [\tilde{\epsilon}^2_n\phi_i^2 f'(a_0^{-\frac{1}{p-2}}\tilde{\epsilon}_n^{-\frac{2}{p-2}}\tilde{u}_n)(1-\eta(a_0^{-\frac{1}{p-2}}\tilde{\epsilon}_n^{-\frac{2}{p-2}}\tilde{u}_n))\\
		&- \tilde{\epsilon}^2_n\phi_i^2 f(a_0^{-\frac{1}{p-2}}\tilde{\epsilon}_n^{-\frac{2}{p-2}}\tilde{u}_n)\eta'(a_0^{-\frac{1}{p-2}} \tilde{\epsilon}_n^{-\frac{2}{p-2}}\tilde{u}_n)]dx    \\
		&\geq \int_{\tilde{\Omega}_{n, M_0} \cap supp\{\phi_i\}} -\tilde{C}'_{M_0} \tilde{\epsilon}^2_n \phi_i^2 dx.
\end{align*}
Note that $a_0^{-\frac{1}{p-2}}\tilde{u}_n\leq \tilde{\epsilon}_n^{\frac{2}{p-2}} M_0$ for $x\in \tilde{\Omega}_{n, M_0}$, by (\ref{2-20-1}) it follows that
\begin{equation*}
\begin{split}
Q_{\lambda_n,\rho_n}(\phi_{i,n}; u_n; \Omega)&=	\int_{\Omega}|\nabla\phi_{i,n} |^2 dx +\lambda_n \int_{\Omega} \phi_{i,n}^2 dx -\int_{\Omega}f_2'(u_n)\phi_{i,n}^2 dx - \rho_n\int_{\Omega}f_1'(u_n)\phi_{i,n}^2 dx\\
 &\leq \int_{\tilde{\Omega}_n}|\nabla \phi_{i} |^2 dx +\lambda_n \tilde{\epsilon}_n^2\int_{\tilde{\Omega}_n} \phi_{i}^2 dx+ \int_{\tilde{\Omega}_{n, M_0} \cap supp\{\phi_i\}} (\rho_n \tilde{C}_{M_0}+ \tilde{C}'_{M_0}) \tilde{\epsilon}^2_n \phi_i^2 dx\\
 &-\rho_n\int_{\tilde{\Omega}_n} \tilde{\epsilon}^2_n \phi_i^2\mu |a_0^{-\frac{1}{p-2}}\tilde{\epsilon}_n^{-\frac{2}{p-2}}\tilde{u}_n|^{p-2} dx+ \rho_n\int_{\tilde{\Omega}_{n, M_0}} \phi_i^2\mu a_0^{-1}|\tilde{u}_n|^{p-2} dx \\
 &\to \int_D |\nabla \phi_i|^2 dx + \tilde{\lambda}\int_D \phi^2 dx- \int_D a_0^{-1}\mu|\tilde{u}|^{p-2}\phi^2 dx\\
 &\leq \int_D |\nabla \phi_i|^2 dx + \tilde{\lambda}\int_D \phi^2 dx- \int_D (p-1)|\tilde{u}|^{p-2}\phi^2 dx<0.
\end{split}
\end{equation*}	
This implies that $m(u_n)\geq k>2$ for sufficiently large $n$, thereby yielding a contradiction. Thus, the claim is valid.

Having established that $\tilde u$ is a finite Morse nontrivial solution of (\ref{tran1}), we can apply either \cite[Theorem 2]{Farina} or \cite[Proposition 2.1]{Yu}. This allows us to conclude that the occurrence of $\tilde{\lambda}=0$ is ruled out, regardless of whether $D$ is a half-space or a whole space.
Consequently, we can assert that $\tilde{\lambda}\in (0, a_0]$. 	

In the sequel, we consider the sequence $\{v_n\}$ defined by (\ref{equivalent}).
Clearly, $v_n$ satisfies
\begin{equation}
\left\{\begin{array}{ll}
-\Delta v_n+ v_n=a_0^{\frac{1}{p-2}}\epsilon_n^{\frac{2p-2}{p-2}}h_{\rho_n}(a_0^{-\frac{1}{p-2}}\epsilon_n^{-\frac{2}{p-2}} v_n)  & {\rm in} \,~ \Omega_n,\\
|v_n(x)|\leq |v_n(0)|=(\frac{\epsilon_n}{\tilde{\epsilon}_n})^{\frac{2}{p-2}}\to \tilde{\lambda}^{-\frac{1}{p-2}}  & {\rm in} \,~ \Omega_n, \\
\displaystyle\frac{\partial v_n}{\partial \nu}=0 \, &{\rm on}\,~\partial \Omega_n.
\end{array}\right.
\end{equation}
Up to a subsequence, we have $v_n\to v$ in $C^2_{loc}(\overline{D})$, where $D$ is either $\mathbb{R}^N$ or a half space $\mathbb{H}$, and $v$ solves
\begin{equation}
\left\{\begin{array}{ll}
-\Delta v+ v=|v|^{p-2}v  & {\rm in} \,~ D,\\
|v(x)|\leq |v(0)|= \tilde{\lambda}^{-\frac{1}{p-2}}  & {\rm in} \,~ D, \\
\displaystyle\frac{\partial v}{\partial \nu}=0 \, &{\rm on}\,~\partial D.
\end{array}\right.
\end{equation}
Arguing as above implies that $m(v)\leq 2$.

More precisely, we distinguish the following three cases:

\begin{itemize}
	\item[(1)]If $\limsup\limits_{n\to +\infty}\frac{dist(P_n, \partial\Omega)}{\tilde{\epsilon}_n}= +\infty$, then $D=\mathbb{R}^N$. Using $m(v)\le2$, by standard regularity arguments, $v\in C^2(\mathbb{R}^N)$, $|v(x)|\to0$ as $|x|\to+\infty$. If $v\ge0$, then using the strong maximum principle and \cite[Theorem 1.1]{EspPet} we know that $v>0, m(v)=1$, and it coincides with the unique radial ground-state solution $U_0$ to $-\Delta u+u=u^{p-1}$ in $\mathbb{R}^N$. If $v$ is sign-changing, then $m(v)=2$  and $v$ admits exactly one local maximum point and one local minimum point in $\Omega$.
 \item[(2)] If there exists a constant $a>0$ such that
    \begin{eqnarray*}
    0<a\leq\limsup_{n\to +\infty}\frac{dist(P_n, \partial\Omega)}{\tilde{\epsilon}_n}< +\infty,
     \end{eqnarray*}
we may assume that $\lim\limits_{n\to +\infty} \frac{dist(P_n, \partial\Omega)}{\tilde{\epsilon}_n}= d>0$. Then, up to subsequences, $v_n\to v $ in $C^2_{loc}(\{x_N>-d\})$, and $v$ weakly solves
\begin{equation*}
        \left\{\begin{array}{ll}
           -\Delta v+ v=|v|^{p-2}v  & {\rm in} \,~ \{x_N>-d\},\\
           |v(x)|\leq |v(0)|= \tilde{\lambda}^{-\frac{1}{p-2}}  & {\rm in} \,~ \{x_N>-d\}, \\
         \displaystyle  \frac{\partial v}{\partial x_N}=0 \, &{\rm on}\,~\{x_N=-d\}.
        \end{array}\right.
        \end{equation*}
        
   Let $\tilde{v}(x)=v(x',x_N-d)$, where $x'=(x_1,x_2,\cdots,x_{N-1})$.
       It is straightforward to verify that $\tilde{v}$ is a finite Morse index solution to the following system:
        \begin{equation*}
        \left\{\begin{array}{ll}
           -\Delta \tilde{v}+ \tilde{v}=|\tilde{v}|^{p-2}\tilde{v}  & {\rm in} \,~ \{x_N>0\},\\
           |\tilde{v}(x)|\leq |\tilde{v}(0,\cdots,0,d)|= \tilde{\lambda}^{-\frac{1}{p-2}}  & {\rm in} \,~ \{x_N>0\}, \\
        \displaystyle   \frac{\partial \tilde{v}}{\partial x_N}=0 \, &{\rm on}\,~ \{x_N=0\}.
        \end{array}\right.
        \end{equation*}
      Next, we extend $\tilde{v}$ by reflection with respect to $\{x_N= 0\}$. Specifically, 
      for $(x',x_N)\in \mathbb{R}^N$, we define
        \begin{equation*}
        	\hat{v}(x',x_N):=
        	\begin{cases}
        		&\tilde{v}(x',x_N) ~~~~~~~~{\rm if}~x_N\geq 0,\\
        		&\tilde{v}(x',-x_N)~~~~~~{\rm if}~x_N<0.
        	\end{cases}
        \end{equation*}
       Consequently, $\hat{v}$ satisfies the equation
        \begin{equation*}
              \left\{\begin{array}{ll}
              -\Delta \hat{v}+ \hat{v}=|\hat{v}|^{p-2}\hat{v}  & {\rm in} \,~ \mathbb{R}^N,\\
              |\hat{v}(x)|\leq |\hat{v}(0,\cdots,0,-d)|= |\hat{v}(0,\cdots,0,d)|=\tilde{\lambda}^{-\frac{1}{p-2}}  & {\rm in} \,~ \mathbb{R}^N, \\
                 \displaystyle    \frac{\partial \tilde{v}}{\partial x_N}=0 \, &{\rm on}\,~ \{x_N=0\}.
              \end{array}\right.
            \end{equation*}
   Thus, $\hat{v}$ is a bounded function that solves the equation $-\Delta \hat{v}+ \hat{v}=|\hat{v}|^{p-2}\hat{v}$ in $\mathbb{R}^N$ in the weak sense. By applying Schauder interior and boundary estimates (see \cite{GT-book}), we conclude that $\hat{v}\in C^{2,\alpha}_{loc}({\mathbb{R}^N})$ for $\alpha\in (0,1)$, and the regularity extends up to the hyperplane $\{x_N=0\}$. Since $\hat{v}$ is symmetric across $\{x_N=0\}$, it follows that the second derivatives of $\hat{v}$ are continuous everywhere, including at the boundary $\{x_N=0\}$. Consequently, by applying elliptic regularity theory and considering that $\hat{v}$ is a solution with finite Morse index, we can show that
$\hat{v}\in C^2(\mathbb{R}^N)$ and $|\hat{v}(x)|\to0$ as $|x|\to+\infty$.
Therefore, $\tilde{v}$ must be a sign-changing solution, a situation that can occur according to \cite[Theorem 1]{McLeod}. In fact, if $\hat{v}$ is a positive solution, then by applying \cite[Theorem 1.1]{EspPet} once again, we deduce that $\hat{v}$ coincides with the unique radial ground-state solution $U_0$ to $-\Delta u+u=u^{p-1}$ in $\mathbb{R}^N$. It is well-known that $U_0$ has a unique global maximum and is a radially strictly decreasing function. This would contradict the fact that $\max\limits_{x\in\mathbb{R}^N} \hat{v}=|\hat{v}(0,\cdots,0,-d)|= |\hat{v}(0,\cdots,0,d)|$.

       \item[(3)] If $\limsup\limits_{n\to +\infty}\frac{dist(P_n, \partial\Omega)}{\tilde{\epsilon}_n}= 0$, then, up to a subsequence, $\Omega_n\to \mathbb{R}^N_+(\{x_N>0\})$.
 By similar arguments as above, $v_n\to v$ in $C^2_{loc}(\mathbb{R}^N_+)$, where $v$ is a finite Morse index of the following equation
       \begin{equation*}
        \left\{\begin{array}{ll}
           -\Delta v+ v=|v|^{p-2}v  & {\rm in} \,~ \{x_N>0\},\\
          \displaystyle \frac{\partial v}{\partial x_N}=0 \, &{\rm on}\,~\{x_N=0\},\\
           |v(y)|\leq |v(0)|= \tilde{\lambda}^{-\frac{1}{p-2}}.
        \end{array}\right.
        \end{equation*}
     Extending $v$ by reflection with respect to $\{x_N=0\}$ and
        defining
        \begin{equation*}
        	\hat{v}(x',x_N):=
        	\begin{cases}
        		v(x',x_N)& ~{\rm if}~~x_N\geq 0,\\
        		v(x',-x_N)&~{\rm if}~~x_N<0.
        	\end{cases}
        \end{equation*}
        Then $\hat{v}$ satisfies
            \begin{equation*}
              \left\{\begin{array}{ll}
              -\Delta \hat{v}+ \hat{v}=|\hat{v}|^{p-2}\hat{v}  & {\rm in} \,~ \mathbb{R}^N,\\
              |\hat{v}(x)|\leq |\hat{v}(0)|  & {\rm in} \, \mathbb{R}^N, ~\\
               \displaystyle  \frac{\partial v}{\partial x_N}=0 \, &{\rm on}\,~\{x_N=0\}.
              \end{array}\right.
            \end{equation*}
Using arguments analogous to those in case (1), we deduce that if $\hat{v}$ is positive, then $m(\hat{v})=1$, and thus $\hat{v}$ coincides with $U_0$. Using \cite[Theorem 1]{McLeod}, it is established that $\hat{v}$ may exhibit sign-changing behavior, and in such instance, $m(v)=2$.
\end{itemize}

    All in all, we get that (ii) holds. Employing similar reasoning to \cite[Theorem 3.1]{EspPet}, we conclude that (iii) and (iv) are also satisfied.
\end{proof}

We now proceed to provide a comprehensive global blow-up analysis. The following result offers a detailed global description of the asymptotic behavior of 
$\{u_n\}$ as $\lambda_n\to+\infty$.
\begin{lemma}\label{globalbe}
There exists $k\in \{1,2\}$ and sequences of points $\{P_n^1\},\cdots,\{P_n^k\}$, such that
\begin{equation}\label{localmaxi}
		|u_n(P_n^i)|= \max_{B_{R_n\lambda_n^{-1/2}}(P_n^i)\cap \overline{\Omega}}|u_n|~for ~some~R_n\to \infty, ~for ~every~i,
	\end{equation}
	\begin{equation}\label{faraway1}
		\lambda_n |P_n^i-P_n^j|^2\to +\infty, ~~\forall i\neq j,~~\mbox{as}~~n\to \infty,
	\end{equation}
and moreover,
    \begin{equation}\label{decay-1}
    	\lim_{R\to \infty}\left(\limsup_{n\to +\infty} \lambda_n^{-\frac{1}{p-2}} \max_{d_n(x)\geq R\lambda_n^{-1/2}}|u_n(x)| \right)=0,
    \end{equation}
 where $d_n(x)=\min\{|x-P_n^i|:i=1,\cdots,k\}$ is the distance function from $\{P_n^1,\cdots, P_n^k\}$ for $x\in \overline{\Omega}$.
\end{lemma}
\begin{proof}
 Take $P_n^1\in \overline{\Omega}$ such that $u_n(P_n^1)=\max\limits_{\overline{\Omega}}|u_n(x)|$. If (\ref{decay-1}) is satisfied for $P_n^1$, then we get $k=1$. It is evident that $P_n^1$ satisfies (\ref{localmaxi}).

Otherwise, if $P_n^1$ does not satisfy (\ref{decay-1}), we suppose that there exists $\delta>0$ such that
    \begin{equation}
    	\lim_{R\to \infty}\left(\limsup_{n\to +\infty} \lambda_n^{-\frac{1}{p-2}} \max_{|x-P_n^1|\geq R\lambda_n^{-1/2}}|u_n(x)| \right)\geq 4\delta.
    \end{equation}
For sufficiently large $R$, up to a subsequence, it holds
\begin{equation}\label{pn2}
	 \lambda_n^{-\frac{1}{p-2}} \max_{|x-P_n^1|\geq R\lambda_n^{-1/2}}|u_n(x)|\geq 2 \delta.
\end{equation}
Let $P_n^2\in \overline{\Omega}\backslash B_{R\lambda_n^{-1/2}}(P_n^1)$ such that
 \begin{equation*}
 	|u_n(P_n^2)|=\max_{\overline{\Omega}\backslash B_{R\lambda_n^{-1/2}}(P_n^1)}|u_n|.
 \end{equation*}
Then, \eqref{pn2} yields that $|u_n(P_n^2)|\to +\infty$ as $n\to +\infty$.

We claim that
 \begin{equation}\label{faraway}
 	\lambda_n|P_n^1-P_n^2|^2\to +\infty.
 \end{equation}
 If \eqref{faraway} is not true, then up to subsequence
 $\lambda_n^{\frac{1}{2}}|P_n^1-P_n^2|\to R'\geq R$.
Define
 \begin{equation*}
 v_{n,1}(x):=a_0^{\frac{1}{p-2}}\lambda_n^{-\frac{1}{p-2}}u_n(\lambda_n^{-1/2}x + P_n^1).
 \end{equation*}
  As in Lemma \ref{localbe}, we can deduce that $ v_{n,1}\to v$ in $C^2_{loc}(\overline{D})$, where $D=\mathbb{R}^N$ or $D$ is a half-space.
 Then, up to subsequences,
 \begin{equation*}
 	\lambda_n^{-\frac{1}{p-2}}u_n(P_n^2)=v_{n,1}(\lambda^{1/2}(P_n^2-P_n^1)) \to v(x'),~|x'|\geq R'>R.
 \end{equation*}
  Since $v(x)\to 0$ as $|x|\to \infty$, taking $R$ larger if necessary, it follows that
  \begin{equation*}
  	|v(x)|\leq a_0^{\frac{1}{p-2}}\delta, ~\forall |x|\geq R.
  \end{equation*}
  This contradicts to \eqref{pn2}, which proves the claim \eqref{faraway}.

In the following, we shall show that
  \begin{equation}\label{yang}
  	|u_n(P_n^2)|=\max_{\overline{\Omega}\cap B_{R_{n,2}\lambda_n^{-1/2}}(P_n^2)}|u_n(x)|, ~{\rm for~some}~ R_{n,2}\to +\infty.
  \end{equation}
Let $\tilde{\epsilon}_{n,2}=a_0^{-\frac{1}{2}}|u_n(P_n^2)|^{-\frac{p-2}{2}}$. Clearly, $\tilde{\epsilon}_{n,2}\to 0$.
 By \eqref{pn2} we get $\tilde{\epsilon}_{n,2}\leq (2\delta)^{-\frac{p-2}{2}}\lambda_n^{-\frac{1}{2}}$.
From \eqref{faraway}, we can assert that
 \begin{equation*}
 	\tilde{R}_{n,2}:=\frac{|P_n^1-P_n^2|}{2\tilde{\epsilon}_{n,2}}\geq \frac{(2\delta)^{\frac{p-2}{2}}}{2}\lambda_n^{\frac{1}{2}}|P_n^1-P_n^2| \to +\infty~{\rm as }~n\to\infty.
 \end{equation*}
On the other hand, for any $x\in B_{\tilde{R}_{n,2}\tilde{\epsilon}_{n,2}}(P_n^2)$ and $R>0$, we have
 \begin{equation*}
 |x-P_n^1|\geq |P_n^2-P_n^1|- |x-P_n^2|\geq \frac{1}{2}|P_n^2-P_n^1|\geq R\lambda_n^{-\frac{1}{2}}
 \end{equation*}
for arbitrarily large $n$. Consequently,
 \begin{equation*}
 	\overline{\Omega}\cap B_{\tilde{R}_{n,2}\tilde{\epsilon}_{n,2}}(P_n^2)\subset \overline{\Omega}\backslash B_{R\lambda_n^{-\frac{1}{2}}}(P_n^1),
 \end{equation*}
 which implies that
 \begin{equation}\label{p2maximum}
 	|u_n(P_n^2)|=\max_{\overline{\Omega}\cap B_{\tilde{R}_{n,2}\tilde{\epsilon}_{n,2}}(P_n^2)}|u_n|.
 \end{equation}

 We define
  \begin{equation*}
\tilde{u}_{n,2}:=a_0^{\frac{1}{p-2}}\tilde{\epsilon}_{n,2}^{\frac{2}{p-2}}u_n(\tilde{\epsilon}_{n,2} x+P_n^2),~~~ x\in \tilde{\Omega}_{n,2}:=\frac{\Omega -P_n^2}{\tilde{\epsilon}_{n,2}}.
 \end{equation*}
Set $d_{n,2}={\rm dist} (P_n^2, \partial \Omega)$. Then, passing to subsequences if necessary,
 \begin{equation*}
 	\frac{\tilde{\epsilon}_{n,2}}{d_{n,2}}\to L_2\in [0,+\infty]~{\rm and}~\tilde{\Omega}_{n,2}\to
 	\begin{cases}
	&\mathbb{R}^N ~~{\rm if}~~L_2=+\infty,\\
	&\mathbb{H}~~{\rm if}~~L_2<+\infty.	
	\end{cases}
 \end{equation*}
 Then $\tilde{u}_{n,2}$ satisfies the following equation
 \begin{equation*}
\left\{\begin{array}{ll}
-\Delta \tilde{u}_{n,2}+ \lambda_n \tilde{\epsilon}_{n,2}^2 \tilde{u}_{n,2}=a_0^{\frac{1}{p-2}}\tilde{\epsilon}_{n,2}^{\frac{2p-2}{p-2}}h_{\rho_n}(a_0^{-\frac{1}{p-2}}\tilde{\epsilon}_{n,2}^{-\frac{2}{p-2}}\tilde{u}_{n,2}) & {\rm in} \,~ \tilde{\Omega}_{n,2},\\
|\tilde{u}_{n,2}(x)|\leq |\tilde{u}_{n,2}(0)|=1  & {\rm in} \,~ \tilde{\Omega}_{n,2}, \\
\displaystyle\frac{\partial \tilde{u}_{n,2}}{\partial \nu}=0 \, &{\rm on}\,~\partial\tilde{\Omega}_{n,2}.
\end{array}\right.
\end{equation*}
 Since $P_n^2$ is a local maximum or a local minimum point, 
by
 Lemma \ref{boundedbelow} we get
\[\frac{\lambda_n}{|u_n(P^2_n)|^{p-2}}\to \tilde{\lambda}^{(2)}\in [0, a_0],~{\rm as}~n\to \infty. \]
 Using the similar argument as in Lemma \ref{localbe}, we deduce $\tilde{\lambda}^{(2)}>0$, namely
 \begin{equation*}
 	\lim_{n\to +\infty}\lambda_n^{\frac{1}{2}}\tilde{\epsilon}_{n,2}>0.
 \end{equation*}
Set
 \begin{equation*}
 	v_{n,2}:=a_0^{\frac{1}{p-2}}\epsilon_n^{\frac{2}{p-2}} u_n(\epsilon_n x+P_n^2)~{\rm for }~x\in \Omega_{n,2}:=\frac{\Omega-P_n^2}{\epsilon_n},
 \end{equation*}
 where $\epsilon_n=\lambda_n^{-\frac{1}{2}}$ is defined in \eqref{equivalent}.
 Up to a subsequence, there exists a function $v^{(2)}\in H^1(D)$ such that $v_{n,2}\to v^{(2)}$ in $C_{loc}^2(\overline{D})$, where $D$ is either $\mathbb{R}^N$ or a half space. Moreover, $v^{(2)}$ solves the following problem
 \begin{equation}
\left\{\begin{array}{ll}
-\Delta v^{(2)}+ v^{(2)}=|v^{(2)}|^{p-2}v^{(2)}  & {\rm in} \,~ D,\\
|v^{(2)}(x)|\leq |v^{(2)}(0)|= (\tilde{\lambda}^{(2)})^{-\frac{1}{p-2}}  & {\rm in} \,~ D, \\
\displaystyle\frac{\partial v^{(2)}}{\partial \nu}=0 \, &{\rm on}\,~\partial D.
\end{array}\right.
\end{equation}
Then, by a similar discussion as above, we conclude that either $\frac{d_{n,2}}{\epsilon_n}\to +\infty$ or $\frac{d_{n,2}}{\epsilon_n}$ remains bounded. Define $R_{n,2}= \tilde{R}_{n,2}\lambda_n^{\frac{1}{2}}\tilde{\epsilon}_{n,2}$. Clearly, $ R_{n,2}\to +\infty$ as $n\to+\infty$. Hence, \eqref{yang} holds.

If (\ref{decay-1}) does not hold, we can apply similar arguments as before to show that there exists $P_n^3$ such that (\ref{localmaxi})-(\ref{faraway1}) are satisfied. For $P_n^i,i=1,2,3$, applying Lemma \ref{localbe} again, we can find $\phi^i_n\in C^\infty_0(\Omega)$ with supp$\phi^i_n\subset B_{R\varepsilon_n}(P_n^i)\cap \overline{\Omega}$ for some $R>0$, such that
\[
\int_\Omega\vert \nabla\phi^i_n\vert~^2dx+\int_\Omega[(\lambda-(p-1)\rho_nu_n^{p-2})(\phi^i_n)^2]dx<0.
\]
In light of (\ref{faraway1}), we observe that $\phi^1_n,\phi^2_n,\phi_n^{3}$ are
mutually orthogonal for sufficiently large $n$, which implies
$\lim\limits_{n\rightarrow+\infty}m(u_n)\ge3$. This leads to a contradiction with the fact that $m(u_n)\leq2$.
\end{proof}

In the subsequent analysis,
we show that $u_n$ exhibits exponential decay away from the blow-up points.
\begin{lemma}\label{decay}
Let $\{P_n^1\},\cdots,\{P_n^k\}$ be given in Lemma \ref{globalbe}.
 Then there exist constants $C_1>0$, $C_2>0$ such that, for some $R>0$,
 	\begin{equation}
		u_n(x)\leq C_1e^{C_1 R}\lambda_n^{\frac{1}{p-2}}\sum_{i=1}^k e^{-C_2\lambda_n^{\frac{1}{2}}|x-P_n^i|}, ~\forall x\in \overline{\Omega}\backslash \cup_{i=1}^k (B_{R\lambda_n^{-\frac{1}{2}}}(P_n^i)\cap \overline{\Omega}),
		\end{equation}
		
\end{lemma}
\begin{proof}
The proof is inspired by \cite[Lemma 2.1]{LZ-2007} and \cite[Theorem 3.2]{EspPet}.
For any $\theta>0$, we set
$$\Omega_\theta:= \{x\in \overline{\Omega}: dist (x,\partial\Omega)<\theta\}, 
$$ 
and define the inner normal bundle
$$(\partial\Omega)_\theta :=\{(x,y):x\in \partial\Omega, y\in (-\theta,0]\nu_x\}, $$
where $\nu_x$ is the unit outer normal of $\partial\Omega$ at $x$.

Since $\partial\Omega$ is a smooth compact submanifold of $\mathbb{R}^N$, by the tubular neighborhood theorem, it follows that there exists a diffeomorphism $\Phi_{Ib}$ from $\Omega_\theta$ onto $(\partial\Omega)_\theta$.
More precisely, for any $x\in \Omega_\theta$, it is easily seen that there exists a unique $\bar{x}\in \partial\Omega$ such that ${\rm dist}(x, \bar{x})= {\rm dist}(x, \partial\Omega)$. Hence, we can define $\Phi_{Ib}(x):=(\bar{x}, -{\rm dist}(x, \bar{x})\nu_{\bar{x}})$ for any $x\in \Omega_\theta$.
Clearly, $\Phi_{Ib}(x)=x$ for $x\in \partial\Omega$.
 
Similarly, let
$$\Omega^\theta := \{x\in \mathbb{R}^N\backslash \Omega, dist (x,\partial\Omega)<\theta\},$$
and define the outer normal bundle
$$(\partial\Omega)^\theta:=\{(x,y):x\in \partial\Omega, y\in [0,\theta)\nu_x\}. $$
Then there exists a diffeomorphism $\Phi_{Ob}:\Omega^\theta \to (\partial\Omega)^\theta$ defined by $\Phi_{Ob}(x):=(\hat{x}, dist(x, \hat{x})\nu_{\hat{x}}), \forall x\in \Omega^\theta $. Here 
 $\hat{x}\in \partial\Omega$ is the unique element in $\partial \Omega$ such that ${\rm dist}(x, \hat{x})= {\rm dist}(x, \partial\Omega)$. Moreover, $\Phi_{Ob}|_{\partial\Omega}= {\rm Identity}$.
 
Consider the reflection $\Phi_C:(\partial\Omega)_\theta\to (\partial\Omega)^\theta$ defined by $\Phi_C((x,y)):=(x,-y)$.
Then, $\Phi:=\Phi_{Ib}^{-1}\circ \Phi_C^{-1}\circ \Phi_{Ob}$ is a diffeomorphism from $\Omega^\theta$ onto $\Omega_\theta$ and $\Phi |_{\partial\Omega}={\rm Identity}$.
Furthermore, we take $x=\Phi(z)=(\Phi_1(z),\cdots,\Phi_N(z))$ for $z\in\Omega^\theta$, $z=\Psi(x)=\Phi^{-1}(x)=(\Psi_1(x),\cdots,\Psi_N(x))$ for $x\in \Omega_\theta$, and
\begin{equation*}
	g_{ij}= \sum_{l=1}^N \frac{\partial \Phi_l}{\partial z_i}\frac{\partial \Phi_l}{\partial z_j},~~
	g^{ij}= \sum_{l=1}^n \frac{\partial \Psi_i}{\partial x_l}\frac{\partial \Psi_j}{\partial x_l}(\Phi (z)).
\end{equation*}
Then $g_{ij}|_{\partial \Omega}= g^{ij}|_{\partial\Omega}=\delta_{ij}$, where $\delta_{ij}$ is the Kronecker symbol.
For simplicity, denote $G=(g^{ij})$ , $g(x)=det(g_{ij})$ and $\hat{u}_n(x)=u_n(\Phi(x))$ for $x\in \Omega^\theta$.
Then $\hat{u}_n$ satisfies 
\begin{equation*}
\left\{\begin{array}{ll}
-L \hat{u}_n+ \sqrt{g}\lambda_n \hat{u}_n=\sqrt{g}h_{\rho_n}(\hat{u}_n) & {\rm in} \,~ \Omega^\theta,\\
\displaystyle\frac{\partial \hat{u}_n}{\partial \nu}=0 \, &{\rm on}\,~\partial \Omega,
\end{array}\right.
\end{equation*}
where
\[L \hat{u}_n= \sum_{i=1}^N \frac{\partial}{\partial x_i}\left(\sqrt{g}\sum_{j=1}^Ng^{ij}\frac{\partial \hat{u}_n}{\partial x_j}\right). \]
Next we define
\begin{equation*}
\bar{u}_n:=
\begin{cases}
u_n(x), x\in\Omega,\\
\hat{u}_n(x), x\in \Omega^\theta,		
\end{cases}
\bar{g}_{ij}:=
\begin{cases}
	\delta_{ij}, x\in \bar\Omega,\\
	g_{ij}, x\in \Omega^\theta,
\end{cases}
\bar{g}^{ij}:=
\begin{cases}
	\delta_{ij}, x\in \bar\Omega,\\
	g^{ij}, x\in \Omega^\theta,
\end{cases}
\end{equation*}
and $\bar{g}:= det(\bar{g}_{ij})$. Let
 $A(x, \xi)=(A_1(x, \xi),\cdots,A_N(x, \xi))$ for $\xi=(\xi_1,\cdots,\xi_N)$ with
\[A_i(x,\xi)=\sqrt{\bar{g}}\sum_{j=1}^N\bar{g}^{ij}\xi_j.\]
Then $\bar{u}_n$ weakly solves
\begin{equation}
	-div\left(A(x,\nabla \bar{u}_n)\right)+ \lambda_n \sqrt{\bar{g}}\bar{u}_n= \sqrt{\bar{g}} g_{\rho_n}(\bar{u}_n)~~{\rm in}~~\overline{\Omega}\cup \Omega^\theta.
\end{equation}


Given that $\Phi$ is a diffeomorphism, the functions  $g_{ij}(x), \frac{\partial g_{ii}}{\partial x_i}$ and $\frac{\partial g_{ij}}{\partial x_i}$ are all smooth on the domain $\Omega^\theta$. Consequently, there exists a constant $C_0>0$ such that $|\frac{\partial g_{ii}}{\partial x_i}|, |\frac{\partial g_{ij}}{\partial x_i}|\leq C_0$ are both bounded by $C_0$ for $x\in \Omega^\theta$.
Moreover, according to the Taylor expansion, as $\theta$ approaches $0$, the functions $g_{ij}$ and $g^{ij}$ tend to $0$, while $g_{ii}$ and $g^{ii}$ tend to $1$, and the determinant $g(x)$ converges to $1$.

For any $\tilde{x}\in \overline{\Omega}$, there exists $\beta>0$ such that $B_\beta (\tilde{x})\subset \overline{\Omega}\cup \Omega^\theta$. For $x\in B_\beta(\tilde{x})\backslash \tilde{x}$, denote $\sigma= |x-\tilde{x}|$.
Then, for any smooth increasing function $\phi(\sigma)$ we have
\begin{align*}
\left|\sum_{i=1}^N \frac{\partial}{\partial x_i}(\sqrt{g}g^{ij})\frac{\partial \phi}{\partial x_j}\right|&= \left|\sum_{i=1}^N  \frac{\partial}{\partial x_i}(\sqrt{g}g^{ij}) \frac{x_j-\tilde{x}_j}{\sigma} \phi'\right|\\
&\leq\sum_{i=1}^N  \bigg| \max_{x\in \overline{\Omega}\cup \Omega^\theta}   \frac{\partial}{\partial x_i}(\sqrt{g}g^{ij}) \bigg|\phi' \frac{diam(\overline{\Omega}\cup \Omega^\theta)}{\sigma}\\
&\le  \frac{C_{\Phi ,\Omega}\phi'}{\sigma}
\end{align*}
for some $C_{\Phi ,\Omega}>0$, where $\phi'=\frac{d\phi(\sigma)}{d\sigma}$. 
Moreover, for $\theta>0$ small enough, there exist small constants $\delta_0, \delta_1$ such that
\begin{align*}
		\left|\sqrt{g} g^{ii}\frac{\partial^2 \phi}{\partial x_i^2}\right|
  &=\left|\sqrt{g} g^{ii} \left(\frac{\phi'}{\sigma}+\phi''\frac{(x_i-\tilde{x}_i)^2}{\sigma^2}- \phi'\frac{(x_i-\tilde{x}_i)^2}{\sigma^3}\right) \right|\\
&\leq (1+\delta_0) \left(\frac{\phi'}{\sigma}+ |\phi''|\frac{(x_i-\tilde{x}_i)^2}{\sigma^2}+\phi'\frac{(x_i-\tilde{x}_i)^2}{\sigma^3}\right),
\end{align*}
and
\begin{equation*}\label{writing}
	\begin{split}
		\left|\sqrt{g} g^{ij}\frac{\partial^2 \phi}{\partial x_j \partial x_i}\right|
   &=\left| \sqrt{g} g^{ij} \left(\phi''\frac{(x_i-\tilde{x}_i)(x_j-\tilde{x}_j)}{\sigma^2}- \phi'\frac{(x_i-\tilde{x}_i)(x_j-\tilde{x}_j)}{\sigma^3}\right) \right|\\
  &\leq |\sqrt{g} g^{ij}| \left(\frac{1}{2} |\phi''| \frac{(x_i-\tilde{x}_i)^2}{\sigma^2}+ \frac{1}{2}\phi'\frac{(x_i-\tilde{x}_i)^2}{\sigma^3}\right)\\
  &+|\sqrt{g} g^{ij}|\left(\frac{1}{2} |\phi''|\frac{(x_j-\tilde{x}_j)^2}{\sigma^2}+ \frac{1}{2}\phi'\frac{(x_j-\tilde{x}_j)^2}{\sigma^3}\right)\\
  &\leq \delta_1 \left(\frac{1}{2} |\phi''| \frac{(x_i-\tilde{x}_i)^2}{\sigma^2}+ \frac{1}{2}\phi'\frac{(x_i-\tilde{x}_i)^2}{\sigma^3}\right)\\
  &+\delta_1 \left(\frac{1}{2} |\phi''|\frac{(x_j-\tilde{x}_j)^2}{\sigma^2}+ \frac{1}{2}\phi'\frac{(x_j-\tilde{x}_j)^2}{\sigma^3}\right), ~{\rm for}~ i\not= j.
	\end{split}
\end{equation*}
 Then
\begin{equation}\label{rongyao}
	\begin{split}
		\left|\sum_{i=1}^N \frac{\partial}{\partial x_i}\left(\sqrt{g}\sum_{j=1}^Ng^{ij}\frac{\partial \phi}{\partial x_j}\right)\right|
	&\leq 2 |\phi''| + \frac{2(N-1)}{\sigma}\phi' + \frac{C_{\Phi ,\Omega}\phi'}{\sigma}.	
	\end{split}
\end{equation}

From now, let us fix $\theta_*\in (0, \theta)$ such that $\frac{1}{2}\leq\sqrt{g}\leq \frac{3}{2}$ and \eqref{rongyao} holds. 
Define $$
A_n^*:=\{x\in \overline{\Omega}\cup \Omega^{\theta_*}: d_n(x)\geq R\lambda_n^{-1/2}\}.
$$
In view of \eqref{decay}, for every $\epsilon\in (0,1)$ small,
to be chosen later, there exist $R^*>0$ and $n_R\in \mathbb{N}$ large such that
\begin{equation*}\label{feng}
	\max_{\{x\in \overline{\Omega} : d_n(x)\geq R\lambda_n^{-1/2}\}}|u_n|\leq \lambda_n^{\frac{1}{p-2}}\epsilon,~\forall R> R^*, ~ \forall n\geq n_R.
\end{equation*}
 Let $A_n: =\{x\in \overline{\Omega}: d_n(x)\geq R\lambda_n^{-1/2}\}$. Clearly,
 \begin{equation*}
 	|u_n(x)|^{p-2}\leq \lambda_n\epsilon^{p-2},~x\in A_n,  ~ \forall n\geq n_R.
 \end{equation*}
 Note that there exists a constant $n_\theta>0$ large enough such that
 for $n>n_\theta$, for any $i\in \{1,\cdots, k\}$, we have
 $B_{\lambda_n^{-1/2}R}(P_n^i)\subset \Omega^{\theta_*}\cup \Omega_{\theta_*}$.
 Then we can deduce that $A_n\subset A_n^*$, for $n>\max\{n_\theta, n_R\}:= n_{\theta R}$.
 From the definition of $\bar{u}_n$, we get
 \[|\bar{u}_n(x)|^{p-2}\leq \lambda_n \epsilon^{p-2}, ~x\in A_n^*,~\forall n> n_{\theta R}. \]
Using $(f_1)-(f_2)$, we conclude that
 \begin{equation}\label{lasa}
 -div\left(A(x,\nabla \bar{u}_n)\right)+ \frac{\lambda_n}{2}\bar{u}_n\leq 0, ~\forall x\in A_n^*
 \end{equation}
 holds for sufficiently small $\epsilon>0$.

For any $x_0\in A_n$ such that $B_r(x_0)\subset A_n^*$, consider the function
\begin{equation*}
	\phi_n(\sigma)=\phi_n(|x-x_0|)=\lambda_n^{\frac{1}{p-2}}\frac{\cosh{(\gamma \lambda_n^{\frac{1}{2}}\sigma})}{\cosh{(\gamma \lambda_n^{\frac{1}{2}}r})}.
\end{equation*}
Clearly, $\phi'(\rho)>0$ and $\phi''(\rho)>0$. By direct computations, we obtain
\begin{equation}\label{chichi}
	\begin{split}
		& 2|\phi_n''|+\frac{2(N-1)}{\sigma}\phi_n'+ \frac{C_{\Phi,\Omega}}{\sigma}\phi_n'- \frac{\lambda_n}{2}\phi_n\\
	&=\lambda_n\phi_n \left(2\gamma^2+ (2(N-1)+C_{\Phi,\Omega})\gamma^2 \frac{\tanh{(\gamma \lambda_n^{\frac{1}{2}}\sigma)}}{\gamma\lambda_n^{\frac{1}{2}}\sigma}-\frac{1}{2} \right)\\
	&\leq \lambda_n\phi_n \left(2\gamma^2+ (2(N-1)+C_{\Phi,\Omega})\gamma^2 -\frac{1}{2} \right)\leq 0,
	\end{split}
\end{equation}
for any $\gamma\in (0, \gamma_*]$, where $\gamma_*=(4N+2C_{\Phi,\Omega})^{-\frac{1}{2}}$. Then, fixing $\gamma\in (0, \gamma_*]$, by \eqref{rongyao} and \eqref{chichi} it follows that
\begin{equation}\label{baile}
	\begin{split}
		div(A(x, \nabla \phi_n))-\frac{\lambda_n}{2}\phi_n =\sum_{i=1}^N \frac{\partial}{\partial x_i}\left(\sqrt{\bar{g}}\sum_{j=1}^N\bar{g}^{ij}\frac{\partial \phi}{\partial x_j}\right)-\frac{\lambda_n}{2}\phi_n
		\leq 0, ~\forall x\in B_r(x_0).
	\end{split}
\end{equation}
In addition, for $x\in \partial B_r(x_0)$, we have $\left(\phi_n -\bar{u}_n \right) \geq \lambda_n^{\frac{1}{p-2}}(1-\epsilon)>0$.
Then, together with (\ref{lasa}) and (\ref{baile}), by the comparison principle (\cite[Theorem 10.1]{PS-2004}) it follows that
$\bar{u}_n\leq \phi_n~~{\rm in}~~B_r(x_0)$,
which implies that $u_n(x_0)\leq \lambda_n^{\frac{1}{p-2}} e^{-\gamma\lambda_n^{\frac{1}{2}}r}$.

Take $r={\rm dist}(x_0, \partial A_n^*)$. We distinguish the following two cases:
\begin{itemize}
    \item [(i)] $dist(x_0, \partial A_n^*)= dist(x_0, B_{R\lambda_n^{-1/2}}(P_n^i))$ for some $i\in \{1,\cdots,k\}$;
	\item [(ii)] $dist(x_0, \partial A_n^*)= dist(x_0, \partial(\overline{\Omega}\cup \Omega^{\theta_*}))$.
\end{itemize}
For case (i), we have $|x_0- P_n^i|= r+ R\lambda_n^{-\frac{1}{2}}$. Then
\[u_n(x_0)\leq e^{\gamma R}\lambda_n^{\frac{1}{p-2}} e^{-\gamma\lambda_n^{\frac{1}{2}}r- \gamma R}=e^{\gamma R}\lambda_n^{\frac{1}{p-2}} e^{-\gamma \lambda_n^{\frac{1}{2}}|x_0-P_n^i|}.\]
For case (ii), we deduce that
\[u_n(x_0)\leq \lambda_n^{\frac{1}{p-2}} e^{-\gamma \lambda_n^{\frac{1}{2}}\theta_*}\leq \lambda_n^{\frac{1}{p-2}} e^{-\gamma \lambda_n^{\frac{1}{2}} \frac{\theta_*}{diam (\Omega)}|x_0-P_n^i|}.\]
Hence, since $x_0$ is arbitrary, we conclude that there exist $C_1, C_2>0$ such that
\begin{equation*}
	u_n(x)\leq C_1 e^{C_1 R}\lambda_n^{\frac{1}{p-2}}\sum_{i=1}^k e^{-C_2\lambda_n^{\frac{1}{2}}|x-P_n^i|}, ~\forall x\in A_n.
\end{equation*}
This completes the proof.
\end{proof}

\section{Proof of Theorem \ref{mainresult}}
In this section, we complete the proof of Theorem \ref{mainresult}.
We have previously established a sequence of mountain pass type critical points $\{u_{\rho_n}\}$ of $J_{\rho_n}$ on $\mathcal{S}_c$, which have uniformly bounded Morse indices. These critical points are constructed for a sequence $\rho_n\to 1^-$. Building on the blow-up analysis conducted in Section 4, we now present the following proposition.

 \begin{proposition}\label{compactness}
	Let $\{u_n\}\subset H^1(\Omega)$ be a sequence of solutions to \eqref{withn}, corresponding to some $\{\lambda_n\}\subset \mathbb{R}$ and $\rho_n\to 1^-$. Suppose that
	\[\int_{\Omega}|u_n|^2 dx=c,~m(u_n)\leq 2,~\forall n, \]
	for some $c>0$, and the energy levels $\{c_n:=J_{\rho_n}(u_n)\}$ are bounded.
	Then, the sequence of solution pairs $\{(u_n, \lambda_n)\}$ is bounded in $H^1(\Omega)\times \mathbb{R}$.
	Moreover, $\{u_n \}$ is a bounded Palais-Smale sequence for the functional $J$ constrained on $\mathcal{S}_c$ at level $c_1$.
\end{proposition}
\begin{proof}
	First, we show that $\{\lambda_n\}$ is bounded. Suppose, by contradiction, that $\lambda_n\to +\infty$.
By Lemma \ref{globalbe}, there exist at most $k$ blow-up limits $\{P_n^1\},\cdots,\{P_n^k\}$ with $k\leq 2$ in $\overline{\Omega}$.	
In the following, we denote by $\{v_n^i \}$ the scaled sequence around $\{P_n^i\}$, and by $v^i$ the limits of $\{v_n^i\}$.
Note that for these blow-up points $\{P_n^1\},\cdots,\{P_n^k\}$, it may hold
\begin{equation*}
	\lambda_n^{\frac{1}{2}}dist(P_n^i,\partial \Omega)\to +\infty ~~{\mbox or}~~\lambda_n^{\frac{1}{2}}dist(P_n^i,\partial \Omega)\to d\geq 0.
\end{equation*}
Without loss of generality,  assume there exists an integer $k_1\in \{0, 1,\cdots, k\}$ such that 
\begin{eqnarray*}
\lambda_n^{\frac{1}{2}}dist(P_n^i,\partial \Omega)\to +\infty~~ \mbox{for every}~~ i\in \{1,\cdots, k_1\},
 \end{eqnarray*}
 and
 \begin{eqnarray*}
\lambda_n^{\frac{1}{2}}dist(P_n^j,\partial \Omega)\to d\geq 0~~ \mbox{for every}~~ j\in \{k_1+1,\cdots, k\}.
 \end{eqnarray*}

On the one hand, we can deduce that for any $R>0$,
	\begin{equation}\label{esitimate}
		\bigg| \lambda_n ^{\frac{N}{2}-\frac{2}{p-2}}\int_{\Omega}u_n^2 dx-\sum_{i=1}^{k_1}\int_{B_R(0)}|v_n^i|^2 dx-\sum_{i=k_1+1}^{k}\int_{B_R(0)\cap {\Omega}_n}|v_n^i|^2 dx \bigg|\to +\infty.
	\end{equation}
	In fact, since $p\in (2_*, 2^*)$, the first term satisfies
	\begin{equation*}
		\bigg| \lambda_n ^{\frac{N}{2}-\frac{2}{p-2}}\int_{\Omega}u_n^2 dx \bigg|=\lambda_n ^{\frac{N}{2}-\frac{2}{p-2}}c\to +\infty.
	\end{equation*}
By Lemma \ref{globalbe}, we have
	\begin{equation*}
		\sum_{i=1}^{k_1}\int_{B_R(0)}|v_n^i|^2 dx\to \sum_{i=1}^k \int_{B_R(0)} |v^i|^2 dx,
	\end{equation*}
	and
	\begin{equation*}
		\sum_{i=k_1+1}^{k}\int_{B_R(0)\cap \overline{\Omega}_n}|v_n^i|^2 dx\to \sum_{i=1}^k \int_{B_R(0)\cap \mathbb{R}_+^N} |v^i|^2 dx,
	\end{equation*}
 which imply that (\ref{esitimate}) holds.

	On the other hand, by Lemma \ref{decay}, there exist constants $C, C'>0$
	such that
\begin{eqnarray}
\begin{aligned}
     &\bigg| \lambda_n ^{\frac{N}{2}-\frac{2}{p-2}}\int_{\Omega}u_n^2 dx
     -\sum_{i=1}^{k_1}\int_{B_R(0)}|v_n^i|^2 dx-\sum_{i=k_1+1}^{k}\int_{B_R(0)\cap {\Omega}_n}|v_n^i|^2 dx \bigg|\\
     &= \lambda_n^{\frac{N}{2}-\frac{2}{p-2}}\int_{\Omega\backslash \cup_{i=1}^{k} (B_{R\lambda_n^{-1/2}}(P_n^i)\cap \Omega)}u_n^2 dx\\
     &\leq C_1  e^{2C_1 R}\lambda_n^{\frac{N}{2}}\sum_{i=1}^{k}\int_{\mathbb{R}^N\backslash \cup_{i=1}^{k}B_{R\lambda_n^{-1/2}}(P_n^i)}e^{-2C_2\lambda_n^{\frac{1}{2}}|x-P_n^i|} dx\\
     &\leq C e^{C R}\sum_{i=1}^{k}\int_{\mathbb{R}^N\backslash B_R(0)} e^{-2C_2|y|} dy \leq C'=C'(R).
\end{aligned}
\end{eqnarray}
Taking $n\to +\infty$, we obtain
a contradiction to \eqref{esitimate}.	
Hence, $\{\lambda_n\}$ is bounded.

Now, we show that $\{u_n\}$ is bounded in $H^1(\Omega)$.
By contradiction, we suppose that $\|u_n\|\to \infty$. Using standard arguments, we have
$\|u_n\|_{L^\infty} \to \infty.$ Take $P_n\in \overline{\Omega}$ such that $|u_n(P_n)|=\|u_n\|_{L^{\infty}}$.
Let $\tilde{\epsilon}_n=a_0^{-\frac{1}{2}}|u_n(P_n)|^{-\frac{p-2}{2}}$ and define
\begin{equation*}
	\tilde{u}_n(x):=a_0^{\frac{1}{p-2}}\tilde{\epsilon}_n^{\frac{2}{p-2}}u_n(\tilde{\epsilon}_nx+ P_n)~~{\rm for}~~ x\in \tilde{\Omega}_n:=\frac{\Omega-P_n}{\tilde{\epsilon}_n}.
\end{equation*}
Clearly, $\tilde{u}_n$ satisfies \eqref{chua}. Using the boundedness of $\{\lambda_n\}$, we get $\lambda_n \tilde{\epsilon}_n^2 \to 0$ as $n\to \infty$.
Then, up to a subsequence, $\tilde{u}_n\to \tilde{u}$ in $C^2_{loc}(\overline{D})$, where $\tilde{u}$ is a finite Morse index solution of
\begin{equation}
\left\{\begin{array}{ll}
-\Delta \tilde{u} =|\tilde{u}|^{p-2}\tilde{u}  & {\rm in} \,~ D,\\
|\tilde{u}(x)|\leq |\tilde{u}(0)|=1  & {\rm in} \,~ D, \\
\frac{\partial \tilde{u}}{\partial \nu}=0 \, &{\rm on}\,~\partial D,
\end{array}\right.
\end{equation}
where $D$ is either $\mathbb{R}^N$ or a half space.
By invoking \cite[Theorem 2]{Farina} and \cite[Proposition 2.1]{Yu} respectively, we conclude that $\tilde{u}\equiv 0$. This
contradicts to the fact that $|\tilde{u}(0)|=1$.
Therefore, we deduce that $\{u_n\}$ is a bounded sequence in $H^1(\Omega)$. Consequently, employing standard arguments, we establish that $\{u_n \}$ is a bounded Palais-Smale sequence for $J|_{\mathcal{S}_c}$ at level $c_1$.
\end{proof}

\begin{proof}[Completion of proof of Theorem \ref*{mainresult}] Let $\{u_n\}$ be the sequence given in Proposition \ref{compactness} for some $c\in (0, c^*)$.
By the compact embedding $H^1(\Omega)\hookrightarrow L^r(\Omega)$ for $r\in [1,2^*)$, and using similar arguments as in Section 3,
we can deduce that $u_{n}\to u$ strongly in $H^1(\Omega)$. This, in turn, implies that $u$ is a mountain pass type normalized solution of \eqref{schrodinger}.
\end{proof}

\section*{Acknowledgements}
Xiaojun Chang was partially supported by NSFC (12471102) and the
Research Project of the Education Department of Jilin Province (JJKH20250296KJ). The research of V.D.~R\u adulescu was supported by the grant ``Nonlinear Differential Systems in Applied Sciences'' of
the Romanian Ministry of Research, Innovation and Digitization, within PNRR-III-C9-2022-I8/22.

\end{document}